\documentclass[10pt,a4paper]{article}
\usepackage[utf8]{inputenc}
\usepackage[english]{babel}
\usepackage[T1]{fontenc}
\usepackage{amsmath}
\usepackage{amsthm}
\usepackage{arydshln}
\usepackage{amsfonts}
\usepackage{amssymb}
\usepackage{amsthm,bbold}
\usepackage{xcolor}
\usepackage{xfrac}
\usepackage{mathtools}
\usepackage{hyperref}
\hypersetup{hidelinks}
\usepackage{tocloft}
\usepackage{mathrsfs}
\usepackage{geometry}
\usepackage{changepage}
\usepackage{stmaryrd}
\usepackage{tikz-cd}
\usepackage[nameinlink]{cleveref}

\theoremstyle{plain}
\newtheorem{theorem}{Theorem}[section]
\newtheorem{definition}[theorem]{Definition}
\newtheorem{lemma}[theorem]{Lemma}
\newtheorem{corollary}[theorem]{Corollary}
\newtheorem{exemple}[theorem]{Example}

\theoremstyle{plain}

\theoremstyle{plain}

\theoremstyle{remark}

\usepackage{thmbox}
\usepackage{geometry}
\usepackage{graphicx}
\graphicspath{{images/}}

\newcommand{\fct}[5]{
    #1:\begin{cases}  #2  &\longrightarrow  #3 \\
         #4  &\longmapsto  #5
\end{cases}}

\newcommand{\R}{\mathbf{R}}

\newcommand{\Z}{\mathbf{Z}}

\newcommand{\N}{\mathbf{N}}

\newcommand{\C}{\mathbf{C}}

\newcommand{\bbP}{\mathbf{P}}

\newcommand{\Oo}{\mathcal{O}}

\title{Extensions of Gorenstein weighted projective 3-spaces and characterization of the primitive curves of their surface sections.}

\author{Bruno Dewer}

\date{}

\begin{document}

\maketitle

\begin{abstract}
We investigate the Gorenstein weighted projective spaces of dimension 3. Given such a space $\bbP$, our first focus is its maximal extension in its anticanonical model $\bbP \subset \bbP^{g+1}$, i.e., the variety $Y\subset \bbP^{g+1+r}$ of largest dimension such that $Y$ is not a cone and $\bbP$ is a linear section of $Y$. In \cite{DS} Thomas Dedieu and Edoardo Sernesi have computed the dimension of $Y$ by cohomological computations on the canonical curves inside $\bbP$. We give an explicit description of $Y$ in the cases where it was not known. Next, we examine the general anticanonical divisors of $\bbP$. These are K3 surfaces, not necessarily primitively polarized. We give a geometric characterization of the curve sections in their primitive polarization.
\end{abstract}

\section{Introduction}

The notion of extendability for projective varieties consists in the following.

\begin{definition} 
A projective variety $X\subset \bbP^N$ is \emph{extendable} if there exists $X_1\subset \bbP^{N+1}$ which is not a cone, such that $X$ is a hyperplane section of $X_1$.

Moreover, if $X_r\subset \bbP^{N+r}$ is not a cone and $X$ can be obtained as the intersection of $X_r$ with a linear subspace of dimension $N$, we say that $X_r$ is an $r$-extension of $X$, or also that $X$ has been extended $r$ times. If moreover there exists no extension of $X$ of larger dimension, we say that $X_r$ is maximal.
\end{definition}

The topic of this paper is the extendability of weighted projective spaces, more precisely those of dimension $3$ that are Gorenstein.

Given four coprime positive integers $a_0,a_1,a_2$ and $a_3$, the weighted projective space $\bbP = \bbP(a_0,a_1,a_2,a_3)$ is $\mathrm{Proj}(R)$ where $R = \C[x_0,x_1,x_2,x_3]$ endowed with the grading $\deg x_i = a_i$ for each $i\in \left\{ 0,1,2,3 \right\}$. By definition, it is Gorenstein if its anticanonical divisor class is Cartier, which holds if and only if all the $a_i$'s divide their sum (see for instance \cite[Theorem 3.3.4]{Do}). As mentioned in \cite{Pr}, among all the weighted projective spaces of dimension $3$, there are exactly $14$ which are Gorenstein.

Assume $\bbP \subset \bbP^{g+1}$ is Gorenstein and embedded by its anticanonical linear system, then its general hyperplane section $S$ is a K3 surface with canonical singularities. The induced polarization $(S,-K_\bbP{}|_{S})$ is of genus $g$, meaning that the general member $\Gamma$ of $-K_\bbP{}|_{S}$ is a canonical curve of genus $g$. As a consequence of this, any extension of $\bbP$ is also an extension of $\Gamma \subset \bbP^{g-1}$.

Consider then a smooth canonical curve $\Gamma \subset \bbP^{g-1}$ obtained as a linear section of $\bbP \subset \bbP^{g+1}$ and the number $\alpha(\Gamma,K_\Gamma)$ introduced in \cite{Lvo}, which can be computed as the corank of the Gauß-Wahl application of the polarization $(\Gamma,K_\Gamma)$; see \color{purple}Definition \ref{def:alpha}\color{black}. In this situation, it follows from \cite{BM} and \cite{Wah} that $\alpha(\Gamma,K_\Gamma)$ is nonzero. Then by \cite[Theorems 2.1, 2.17]{CDS} we have $\dim Y = 1 + \alpha(\Gamma,K_\Gamma)$, for $Y$ a maximal extension of $\Gamma$. This allows us to know exactly how many times $\bbP$ can be extended, by \cite[Corollary 6.4]{DS}.

The full list of the Gorenstein weighted projective $3$-spaces, as well as the maximal extensions which are known from \cite{DS}, is given below.

Notice that the space $\bbP(1,1,1,3,5,5,5)$ is denoted by $\bbP(1^3,3,5^3)$, as the weights $1$ and $5$ each appear three times, and the weight $3$ appears once. From now on, we will adopt this notation for brevity. We also adopt the following convention: when $\bbP$ is not extendable, we say that it is its own maximal extension.
$$
\begin{array}{|l|l|l|}
\hline
\bbP & \text{extendable?} & \text{maximal extension} \\
\hline
\bbP(1,1,1,1) & \text{no} & \text{itself} \\
\bbP(1,1,1,3) & \text{no} & \text{itself} \\
\bbP(1,1,4,6) & \text{no} & \text{itself} \\
\bbP(1,2,2,5) & \text{yes} & \text{sextic hypersurface of } \bbP(1^3,3,5^3) \\
\bbP(1,1,2,4) & \text{no} & \text{itself} \\
\bbP(1,3,4,4) & \text{yes} & \text{quartic hypersurface of } \bbP(1^4,3^4) \\
\bbP(1,1,2,2) & \text{no} & \text{itself} \\
\bbP(2,3,3,4) & \text{yes} & \text{cubic hypersurface of } \bbP(1^5,2^5) \\
\hdashline[8pt/3pt]
\bbP(1,4,5,10) & \text{yes} & \text{was not known} \\
\bbP(1,2,6,9) & \text{yes} & \text{was not known} \\
\bbP(1,2,3,6) & \text{no} & \text{itself} \\
\bbP(1,3,8,12) & \text{yes} & \text{was not known} \\
\bbP(1,6,14,21) & \text{yes} & \text{was not known} \\
\bbP(2,3,10,15) & \text{yes} & \text{was not known} \\
\hline
\end{array}
$$

\begin{definition} \label{def:moduli stack polarized K3}
Let $\mathcal K_g^{i}$ be the moduli stack of the polarized surfaces $(S,L)$ with $S$ a K3 surface and $L$ an ample line bundle (equivalently, an ample Cartier divisor) on $S$ such that the general member of $|L|$ is a genus $g$ curve and the index of $L$ in $\mathrm{Pic}(S)$ is equal to $i$. In other words, $L \simeq (L_\mathrm{prim})^{\otimes i}$ with $L_\mathrm{prim}$ a primitive line bundle on $S$. 
\end{definition}

We refer to \cite[\S 5]{Huy} for the construction of moduli stacks of polarized K3 surfaces.

\vspace{.2cm}
For all $g$ and $i$ we consider the function $\alpha : \mathcal K_g^{i} \to \mathbf Z$ given by $\alpha(S,L) = \alpha(\Gamma',K_{\Gamma'})-1$, where $\Gamma'$ is a general member of $|L|$. Given $\bbP$ a Gorenstein weighted projective space of dimension $3$ and $S$ a general anticanonical divisor of $\bbP$, let $g$ and $i_S$ respectively denote the genus and the index of the induced polarization $(S,-K_\bbP|_S)$. This polarization is then a member of the moduli stack $\mathcal K_g^{i_S}$. T.~Dedieu and E.~Sernesi have computed \cite[Proposition 6.2]{DS} that in each of the 14 cases, there is a constant $\alpha_g^{i_S}$ such that $\alpha$ takes the value $\alpha_g^{i_S}$ on a dense open subset of $\mathcal K_g^{i_S}$. The first 8 cases of the list above are those for which $\alpha(S,-K_\bbP|_S) = \alpha_g^{i_S}$, and we are going to examine the ones for which this equality doesn't hold.

The core results of this paper, proven in \hyperlink{section 4}{Section 4} and \hyperlink{section 5}{Section 5}, are summarized in the two following theorems.

\begin{theorem} \label{thm:maximal extensions}
Assume that the polarization $(S,-K_\bbP|_S)$ is not general in $\mathcal K_g^{i_S}$, in the sense that 
$$
\alpha(S,-K_\bbP|_S) > \alpha^{i_S}_g.
$$
Then $\bbP$ is one of the last six items of the list given above. Each of them admits a maximal extension $Y$ which has a description as follows.
$$
\begin{array}{|l|l|l|}
\hline
\bbP & Y & \textsl{dim}(Y) \\
\hline & & \\[-10pt]
\bbP(1,4,5,10) & \textsl{nongeneral quintic of } \bbP(1^3,2,4^3) & 5 \\
\bbP(1,2,6,9) & \textsl{nongeneral } 10 \textsl{-ic of } \bbP(1^2,3,5,9^2) & 4 \\
\bbP(1,2,3,6) & \bbP(1,2,3,6) & 3 \\
\bbP(1,3,8,12) & \textsl{nongeneral } 9 \textsl{-ic of } \bbP(1^2,3,4,8^2) & 4 \\
\bbP(1,6,14,21) & \textsl{nongeneral heptic of } \bbP(1^2,2,3,6^2) & 4 \\
\bbP(2,3,10,15) & \textsl{codim. } 2 \textsl{ complete intersection in a} & 5 \\
& \bbP(1^2,2,3,5^3)\textsl{-bundle over } \bbP^1 & \\
\hline
\end{array}
$$
\end{theorem}

Next, we consider $C$ a general member of the primitive divisor class $-\frac{1}{i_S}K_\bbP|_S$. We focus on the same cases as those of \color{purple}Theorem \ref{thm:maximal extensions}\color{black}, and to give an insight on the geometry of $S$, we provide a geometric characterization of $C$.

\begin{theorem}\label{thm:primitive curves characterization}
Let $\bbP$ be a Gorenstein weighted projective space and $S$ a general anticanonical divisor of $\bbP$ such that $\alpha(S,-K_\bbP|_S) > \alpha^{i_S}_g$. Then the general member $C$ of $|-\frac{1}{i_S}K_\bbP|_S|$ is as follows.
$$
\begin{array}{|l|l|}
\hline
\bbP & C \\
\hline & \\[-10pt]
\bbP(1,4,5,10) & \textsl{plane quintic with a total inflection point} \\
\bbP(1,2,6,9) & \textsl{smooth hyperelliptic curve of genus } 4\\
\bbP(1,2,3,6) & \textsl{normalization of a plane sextic with an oscnode} \\
\bbP(1,3,8,12) & \textsl{trigonal curve of genus } 7 \textsl{ with a total ramification point} \\
\bbP(1,6,14,21) & \textsl{blowup of a plane } 21\textsl{-ic curve at 8 heptuple points} \\
\bbP(2,3,10,15) & \textsl{normalization of a nodal } 6\textsl{-gonal curve of genus } 16 \textsl{ such that the} \\
 & g_6^1 \textsl{ has two members of the form } 6p \textsl{ and } 2p_1 + 2p_2 + 2p_3 \textsl{ respectively} \\
\hline
\end{array}
$$
Conversely, for all items on the list except $\bbP(1,2,3,6)$, any curve with the given description is isomorphic to a member of $|-\frac{1}{i_S}K_\bbP|_{S'}|$ for some K3 surface $S'\in |-K_\bbP|$.
\end{theorem}

The organization of the article is as follows. In \hyperlink{section 2}{Section 2} we go over some information and definitions about the $3$-dimensional Gorenstein weighted projective spaces. In \hyperlink{section 3}{Section 3} we introduce a birational model of $\bbP$ which realizes $S$ as a nongeneral anticanonical divisor of another weighted projective $3$-space. This allows us in \hyperlink{section 4}{Section 4} to express $\Gamma$ as a complete intersection of two surfaces of different degrees and to prove \color{purple}Theorem 1.3\color{black}, constructing the maximal extension $Y$ as a hypersurface in a larger weighted projective space in all cases, except one: this particular case requires additional work. In \hyperlink{section 5}{Section 5} we consider the primitive polarization of $S$ and give a geometric characterization of the general curve $C$ in the linear system $|-\frac{1}{i_S}K_\bbP{}|_S|$, hence proving \color{purple}Theorem 1.4\color{black}.

\subsection{Notations and conventions}

Let us list here some notations and conventions that we will use throughout this paper.

\begin{enumerate}
    \item[$\bullet$] In the literature, the term \emph{K3 surface} commonly refers to a smooth surface $S$ with trivial canonical divisor and $h^1(S,\mathcal O_S) = 0$.

    \vspace{.2cm}
    However, these two conditions also make sense for a normal, possibly singular, surface. As already mentioned in the introduction above, the surface $S$ we obtain as a general anticanonical divisor of a Gorenstein weighted projective $3$-space $\bbP$ is expected to have canonical singularities (to be more precise, ADE singularities). However, it is normal and satisfies the conditions $K_S = 0$ and $h^1(S,\mathcal O_S)=0$. As other authors did before, we refer to such a surface $S$ as a K3 surface with ADE singularities, and we sometimes simply say K3 surface.

    \vspace{.2cm}
    Note that, in this case, one may consider a crepant resolution $S' \to S$ where $S'$ is a smooth K3 surface. This will be useful later on, for instance in our \color{purple}Lemma \ref{lem:hyperelliptic curve in K3}\color{black}.

    \item[$\bullet$] The exponential notation for weighted projective spaces, which was briefly mentioned in the introduction above, consists in the following. If a weight appears multiple times in a weighted projective space, then it may be denoted with an exponent; for instance, $\bbP(1^3,3,5^3)$ is $\bbP(1,1,1,3,5,5,5)$.

    \vspace{.2cm}
    Furthermore, when a weighted projective space is introduced with a choice of homogeneous weighted coordinates, then the coordinates may be written as a subscript. As an example, the space $\bbP(1,1,1,2)$ endowed with coordinates $[u_0:u_1:u_2:v]$ is written as $\bbP(1,1,1,2)_{[u_0:u_1:u_2:v]}$. Note that one letter is used for each weight: here, the $u_i$'s have degree $1$, while $v$ has weight $2$. When a letter appears multiple time, we allow ourselves to write them as a single letter in bold. According to this and to the exponential notation, we may write $\bbP(1^3,2)_{[\mathbf u:v]}$ instead of $\bbP(1,1,1,2)_{[u_0:u_1:u_2:v]}$.

    \item[$\bullet$] On a variety $X$ endowed with a line bundle $L$, we denote by $|L|$ the linear system $\bbP(H^0(X,L))$ of effective Cartier divisor associated with $L$; this notation has already been used in \color{purple}Definition \ref{def:moduli stack polarized K3}\color{black}. This is inspired by the classical notation $|D|$ for the linear system associated with a divisor $D$.
    
\end{enumerate}

\section{The Gorenstein weighted projective spaces of dimension $3$} \hypertarget{section 2}{}

This section is entirely dedicated to results which were already known facts. We provide proofs for the sake of exhaustiveness, and we will refer to those results in the remaining of the article.

We refer to \cite[§5, §6]{I} for basic facts, generalities and notations about weighted projective spaces. They are normal, $\mathbf Q$-factorial, Cohen-Macaulay projective schemes with quotient singularities which are natural generalizations of standard projective spaces.

Let $\bbP = \bbP(a_0,a_1,a_2,a_3)$ be a weighted projective $3$-space. Its anticanonical divisor class is that of degree $a_0+a_1+a_2+a_3$ surfaces and its Picard group is generated by $[\Oo_\bbP(l)]$ with $l = \mathrm{lcm}(a_0,a_1,a_2,a_3)$. Hence $K_\bbP$ is Cartier if and only if $a_i$ is a divisor of $a_0+a_1+a_2+a_3$ for all $i\in \left\{ 0,1,2,3 \right\}$.

\vspace{.2cm}
\noindent
\hypertarget{Remark}{}
\textbf{Remark.} If $\bbP$ is not Gorenstein, then the canonical divisor $K_\bbP$ is not Cartier around at least one \emph{coordinate point}, i.e., a point given by the vanishing of $3$ out of the $4$ homogeneous coordinates. As an example, consider $\bbP = \bbP(1,2,4,5)_{[u:v:s:t]}$, with $\mathcal O_\bbP(-K_\bbP) = \mathcal O_\bbP(12)$. Since $5$ is not a divisor of $12$, then any local section of $\mathcal O_\bbP(12)$ around $p_t = [0:0:0:1]$ is a quotient of polynomials with numerator a sum of monomials, all of which admit at least one of the coordinates $u,v,s$ as a factor. Hence such a local section vanishes at $p_t$.

\vspace{.2cm}
More generally, on a weighted projective $3$-space $\bbP = \bbP(a_0,a_1,a_2,a_3)$, the intersection form of $\mathbf Q$-Cartier divisors is given by
$$
\mathcal O_\bbP(1)^3 = \frac{1}{a_0a_1a_2a_3}.
$$
Moreover, up to isomorphism, we may assume that any three of the four weights $a_0,a_1,a_2,a_3$ are coprime. This condition is called \emph{well-formedness}.

\vspace{.2cm}
Assume that the canonical divisor $K_\bbP$ is Cartier, i.e., that $\bbP$ is Gorenstein. In this situation, we will prove in \color{purple}Theorem \ref{thm:anticanonical very ample} \color{black}that $\bbP$ is embedded in a projective space $\bbP^{g+1}$ by its anticanonical linear system. Let $\Gamma \subset \bbP^{g-1}$ be a curve section of $\bbP$ cut out by two general hyperplanes. The adjunction formula yields $K_\Gamma = -K_\bbP{}|_\Gamma.$ Hence $\Gamma$ in $\bbP^{g-1}$ is a canonical curve, and $\bbP$ can only be extended finitely many times, by \color{purple}Theorem \ref{thm:canonical curve max extension} \color{black}and \color{purple}Corollary \ref{cor:clifford index}\color{black}.

We list all the Gorenstein weighted projective spaces of dimension $3$ in \color{purple}\hyperlink{Table 1}{Table 1} \color{black}below, together with the following information. If $a_i$ are the weights of $\bbP$, $l=\mathrm{lcm}(a_0,a_1,a_2,a_3)$ and $\sigma = a_0 + a_1 + a_2 + a_3$, then the index $i$ of $-K_\bbP$ in $\mathrm{Pic}(\bbP)$ is equal to $\frac{\sigma}{l}$.

$$\hypertarget{Table 1}{}
\begin{array}{|l|l|l|l|l|}
\hline
\# & \bbP & l & \sigma & i \\
\hline & & & & \\[-10pt]
1 & \bbP(1,1,1,1) & 1 & 4 & 4 \\
2 & \bbP(1,1,1,3) & 3 & 6 & 2 \\
3 & \bbP(1,1,4,6) & 12 & 12 & 1 \\
4 & \bbP(1,2,2,5) & 10 & 10 & 1 \\
5 & \bbP(1,1,2,4) & 4 & 8 & 2 \\
6 & \bbP(1,3,4,4) & 12 & 12 & 1 \\
7 & \bbP(1,1,2,2) & 2 & 6 & 3 \\
8 & \bbP(2,3,3,4) & 12 & 12 & 1 \\
9 & \bbP(1,4,5,10) & 20 & 20 & 1 \\
10 & \bbP(1,2,6,9) & 18 & 18 & 1 \\
11 & \bbP(1,2,3,6) & 6 & 12 & 2 \\
12 & \bbP(1,3,8,12) & 24 & 24 & 1 \\
13 & \bbP(1,6,14,21) & 42 & 42 & 1 \\
14 & \bbP(2,3,10,15) & 30 & 30 & 1 \\
\hline
\end{array}
$$
\begin{center}
\footnotesize{Table 1}
\end{center}

\subsection{Extendability of $\Gamma$ and $\bbP$}

For each $\bbP$ in \color{purple}\hyperlink{Table 1}{Table 1}\color{black}, if $S$ is a general anticanonical divisor, then the couple $(S,-K_\bbP|_S)$ represents an element of the moduli stack $\mathcal K_g^{i_S}$; see \color{purple}Definition \ref{def:moduli stack polarized K3}\color{black}. We focus here on the last $6$ examples ($\#9$ to $\#14$) which are mentioned in \color{purple}Theorem \ref{thm:maximal extensions}\color{black}. A maximal extension of $\bbP$ in these cases was not known so far.

\begin{definition} \label{def:alpha}
Let $X\subset \bbP^N$ be a projective variety, and $L = \Oo_{\bbP^N}(1)|_X$. We introduce
$$
\alpha(X,L) = h^0(N_{X/\bbP^N}\otimes L^{-1}) - N - 1.
$$
So that, if $X = \Gamma$ is a canonical curve in $\bbP^{g-1}$, one has $\alpha(\Gamma,L) = \alpha(\Gamma,K_\Gamma)$.
\end{definition}

\begin{lemma}
In the case where $\Gamma$ is a canonical curve, it holds that $\alpha(\Gamma,K_\Gamma) = \mathrm{cork}(\Phi_{K_\Gamma})$ where $\Phi_{K_\Gamma}$ is the Gauß-Wahl map of the polarization $(\Gamma,K_\Gamma)$.
\end{lemma}

We refer to §2 in \cite{BM}, and \cite{Wah} for the definition of the Gauß-Wahl map of a polarized curve and to \cite[Lemma 3.2]{CDS} for a proof of this lemma. The value of $\alpha(\Gamma,K_\Gamma)$ for $\Gamma$ a general curve linear section of any Gorenstein weighted projective $3$-space have been computed by T.~Dedieu and E.~Sernesi and we display these values in the relevant cases in \color{purple}\hyperlink{Table 2}{Table 2} \color{black}below. This allows us to know the dimension of any maximal extension of $\Gamma$ by the following theorem.

\begin{theorem}[\cite{CDS}, Theorem 2.1 \& Corollary 5.5] \label{thm:canonical curve max extension}
Let $\Gamma \subset \bbP^{g-1}$ be a smooth canonical curve with $g\geq 11$ and such that $\mathrm{Cliff}(\Gamma)>2$. Then $\Gamma$ is extendable only $\alpha(\Gamma,K_\Gamma)$ times, i.e., there exists $Y$ in $\bbP^{g-1+\alpha(\Gamma,K_\Gamma)}$ such that $\dim(Y) = 1+\alpha(\Gamma,K_\Gamma)$, which is a maximal extension of $\Gamma$.

In addition, there exists a maximal extension $Y$ of $\Gamma$ which is universal, meaning that for each surface extension $S$ of $\Gamma$, there is a unique $g$-plane $\Lambda \subset \bbP^{g-1+\alpha(\Gamma,K_\Gamma)}$ such that $\Gamma \subset \Lambda$ and $S = Y\cap \Lambda$.
\end{theorem}

We make the comment that the universal extension of a smooth canonical curve of genus at least $11$ and Clifford index at least $3$ is unique up to isomorphism, by the construction provided in \cite[\S 5]{CDS}.

\vspace{.2cm}
Let us consider $\Gamma$ a general linear curve section of $\bbP \subset \bbP^{g+1}$ where $\bbP$ is a Gorenstein weighted projective $3$-space. Hence, $\Gamma$ is a general hyperplane section of a K3 surface $S\subset \bbP$. Such a surface only has isolated singularities, so by Bertini's Theorem, $\Gamma$ is smooth. The values of $g$ are listed in \color{purple}\hyperlink{Table 2}{Table 2} \color{black}below, and in each case we have $g\geq 11$ and $\mathrm{Cliff}(\Gamma)>2$ by \color{purple}Corollary \ref{cor:clifford index}\color{black}, to the effect that \color{purple}Theorem \ref{thm:canonical curve max extension} \color{black}applies to $\Gamma$.

As a consequence, we know that any extension of $\bbP$ of dimension $1+\alpha(\Gamma,K_\Gamma)$ is maximal. In particular, in any case for which $\alpha(\Gamma,K_\Gamma) = 2$, the threefold $\bbP$ is not extendable and it is the universal extension of $\Gamma$.

We list in \color{purple}\hyperlink{Table 2}{Table 2} \color{black}examples \#9 to \#14 coupled with the data of $i_S$, the genera of $\Gamma$ and $C$, where $\Gamma$ is a general member of $-K_\bbP|_S$ and $C$ is a general member of $-\frac{1}{i_S}K_\bbP|_S$. The value of $\alpha(\Gamma,K_\Gamma)$ and the datum of the singular points of $S$ are also provided.

$$
\hypertarget{Table 2}{}
\begin{array}{|l|l|l|l|l|l|l|}
\hline
\# & \bbP & i_S & g=g(\Gamma) & g(C) & \alpha(\Gamma,K_\Gamma) & \mathrm{Sing}(S) \\
\hline & & & & & & \\[-10pt]
9 & \bbP(1,4,5,10) & 2 & 21 & 6 & 4 & A_1,2A_4 \\
10 & \bbP(1,2,6,9) & 3 & 28 & 4 & 3 & 3A_1,A_2 \\
11 & \bbP(1,2,3,6) & 2 & 25 & 7 & 2 & 2A_1,2A_2 \\
12 & \bbP(1,3,8,12) & 2 & 25 & 7 & 3 & 2A_2,A_3 \\
13 & \bbP(1,6,14,21) & 1 & 22 & 22 & 3 & A_1,A_2,A_6 \\
14 & \bbP(2,3,10,15) & 1 & 16 & 16 & 4 & 3A_1,2A2,A_4 \\
\hline
\end{array}
$$
\begin{center}
\footnotesize{Table 2}
\end{center}

In each of those cases, we will see in \color{purple}Theorem \ref{thm:anticanonical very ample} \color{black}that the anticanonical divisor $-K_\bbP$ is very ample. Hence it embeds $\bbP$ in $\bbP^{g+1}$ as a variety of degree $(-K_\bbP)^3 = 2g-2$. In that model, recall that $\bbP$ can only be extended $\alpha(\Gamma,K_\Gamma)-2$ times. The only one which is not extendable in our list is $\bbP(1,2,3,6)$.

\subsection{The very ampleness of $-K_\bbP$}

Let us prove now that $-K_\bbP$ is very ample in the cases $\#9$ to $\#14$ given above. For all those spaces, this requires the nonhyperellipticity of $\Gamma$. The proof of the nonhyperellipticity of $\Gamma$ and the very ampleness of $K_\Gamma$ requires that we state first the following lemma.

\begin{lemma} \label{lem:hyperelliptic curve in K3}
Let $S$ be a K3 surface (possibly with ADE singularities) and $\Gamma$ a smooth curve of genus at least $2$ contained in the smooth locus of $S$. If $\Gamma$ is hyperelliptic, then there exists a line bundle $\mathcal J$ on $S_\mathrm{smooth}$ such that $|\mathcal J|_\Gamma|$ is a $g^1_2$ (i.e., a pencil of degree $2$ divisors).
\end{lemma}

A proof of this lemma relies on \cite[Main Theorem]{GL}, by the fact that $\Gamma$ is hyperlliptic iff. $\mathrm{Cliff}(\Gamma) = 0$, in which case there is a line bundle on $S$ whose restriction to $\Gamma$ is a pencil of degree $2$. The main Theorem of \cite{GL} is stated for smooth projective curves on smooth projective K3 surfaces; when $\Gamma$ is a smooth projective curve contained in the smooth locus of a singular K3 surface $S$ (which is our setting in \color{purple}Lemma \ref{lem:hyperelliptic curve in K3} \color{black}above), the argument holds by taking a resolution $S' \to S$ with $S'$ a smooth K3 surface. 

We will apply this lemma together with the known fact that, for a given curve $\Gamma$ of genus $g\geq 2$, $K_\Gamma$ is very ample if and only if $\Gamma$ is nonhyperelliptic (see \cite[Proposition IV.5.2]{Har}).

\begin{theorem} \label{thm:canonical very ample}
Let $\bbP$ be one of the Gorenstein weighted projective $3$-spaces listed in \color{purple}\hyperlink{Table 2}{Table 2}\color{black}. Then the general linear curve section $\Gamma \subset \bbP$ in nonhyperelliptic and $K_\Gamma$ is very ample.
\end{theorem}

Apart from \color{purple}Lemma \ref{lem:hyperelliptic curve in K3}\color{black}, the proof of \color{purple}Theorem \ref{thm:canonical very ample} \color{black}requires information which will be given later. Therefore, let us postpone the proof; it can be found right after \color{purple}\hyperlink{Table 5}{Table 5}\color{black}. We now state the following corollary of \color{purple}Theorem \ref{thm:canonical very ample}\color{black}:

\begin{corollary} \label{cor:clifford index}
In the setting of \color{purple}Theorem \ref{thm:canonical very ample}\color{black}, the Clifford index of the curve $\Gamma$ is strictly larger than $2.$
\end{corollary}

\begin{proof}
Since the anticanonical model $\bbP\subset \bbP^{g+1}$ satisfies the $N_2$ property, as stated in \cite[Proposition 6.1]{DS}, so does the curve $\Gamma$. It follows by the appendix of \cite{GL84} that $\mathrm{Cliff}(\Gamma)>2$.
\end{proof}

 Subsequently, we are able to prove that the anticanonical divisor $-K_\bbP$ is very ample, for each weighted projective space $\bbP$ listed in Table 2.

\begin{lemma} \label{lem:very ample characterization}
Let $X$ be a projective variety and $D$ an ample Cartier divisor on $X$. Let $A_n = H^0(X,\Oo_X(nD))$ for all $n\in \N$ and $A = \oplus_{n\geq 0} A_n$. If $A$ is generated by $A_1$ as a $\C$-algebra, then $D$ is very ample.
\end{lemma}

\begin{proof}
Let $q$ be a positive integer such that $Z = qD$ is very ample. It induces an embedding of $X$ into a projective space, thus
$$
X \simeq \mathrm{Proj}(A^{(q)}) \simeq \mathrm{Proj}(A),
$$
where $A^{(q)}$ is the graded ring such that $(A^{(q)})_d = A_{dq}$. Let 
$$
\varphi : X \dashrightarrow \bbP(H^0(X,\Oo_X(D)))
$$ 
be the map induced by $D$. If $L_n$ is the image of the multiplication map 
$$
H^0(X,\Oo_X(D))^{\otimes n} \to H^0(X,\Oo_X(nD)),
$$
then it holds that $\varphi(X) = \mathrm{Proj}(L)$ with $L = \oplus_{n\geq 1}L_n$. The condition that $A$ is generated by $A_1$ is equivalent to $A = L$, and therefore it yields $\varphi(X) = \mathrm{Proj}(A) \simeq X$.
\end{proof}

\begin{theorem} \label{thm:anticanonical very ample}
Let $\bbP = \bbP(a_0,a_1,a_2,a_3)$ be one of the Gorenstein weighted projective $3$-spaces listed in \color{purple}\hyperlink{Table 2}{Table 2}\color{black}. Then $-K_\bbP$ is very ample.
\end{theorem}

\begin{proof}
Let $S$ be a general member of $|-K_\bbP|$ and $\Gamma$ a general member of $|-K_\bbP|_S|$. We first prove that $-K_\bbP|_S$ is very ample. Thanks to \color{purple}Lemma \ref{lem:very ample characterization}\color{black}, it is enough to prove that $H^0(S,\Oo_S(-K_\bbP|_S))^{\otimes n} \to H^0(S,\Oo_S(-nK_\bbP|_S))$ is onto for all $n\geq 1$. It is naturally the case for $n=1$. Assume now that $n>1$ and that it holds up to rank $n-1$. Consider the following commutative diagram

\begin{center}
\begin{tikzcd}
H^0(\Gamma,\Oo_\Gamma(K_\Gamma))^{\otimes n} \arrow[r,twoheadrightarrow] & H^0(\Gamma,\Oo_\Gamma(nK_\Gamma)) \\
H^0(S,\Oo_S(-K_\bbP|_S))^{\otimes n} \arrow[r] \arrow[u,twoheadrightarrow] & H^0(S,\Oo_S(-nK_\bbP|_S)) \arrow[u,twoheadrightarrow] \\
H^0(S,\Oo_S(-K_\bbP|_S))^{\otimes n-1} \arrow[r,twoheadrightarrow] \arrow[u,hookrightarrow,"\otimes f"] & H^0(S,\Oo_S(-(n-1)K_\bbP|_S)) \arrow[u,hookrightarrow,"\cdot f"]
\end{tikzcd}
\end{center}

\noindent
Here, $f$ is a global section of $\Oo_S(-K_\bbP|_S)$ such that $\Gamma = \left\{ f=0\right\}$ on $S$. The bottom arrow is onto by the induction hypothesis.

The right column is the restriction exact sequence. The curve $\Gamma$ lies on the surface $S$, which is itself embedded in $\bbP^g$. By \cite[Lemma 2.1, Proposition 2.1]{Sho}, we have $h^1(\bbP^g,\mathcal I_{\Gamma|\bbP^{g}}(n)) = 0$ for $n\geq 1$. Besides, by $\cite[6.6]{SD}$, we have $h^1(\bbP^g,\mathcal I_{S|\bbP^g}(n)) = 0$ for $n\geq 1$. It follows that $h^1(S,\mathcal I_{\Gamma|S}(n)) = 0$ for $n\geq 1$, hence the top vertical arrows are surjective. This argument works since $g(\Gamma)\geq 3$ and $\Gamma$ is nonhyperelliptic, ensuring that $\Gamma$ is projectively normal. The surjectivity of the top horizontal arrow follows from the surjectivity of $H^0(\bbP^g,\mathcal O_{\bbP^g}(n)) \to H^0(\Gamma,\mathcal O_\Gamma(nK_\Gamma))$, by the vanishing $h^1(\bbP^g,\mathcal I_{\Gamma|\bbP^g}(n)) = 0$, and the identification
$$
H^0(\Gamma,\mathcal O_\Gamma(K_\Gamma))^{\otimes n} \simeq H^0(\bbP^g,\mathcal O_{\bbP^g}(1))^{\otimes n} = H^0(\bbP^g,\mathcal O_{\bbP^g}(n)).
$$
We refer to \color{purple}Theorem \ref{thm:canonical very ample} \color{black}for a case-by-case proof that $\Gamma$ is nonhyperelliptic and $K_\Gamma$ is very ample.

As the vertical sequence on the right is exact, the surjectivity of the middle arrow follows from the surjectivity of the bottom and the top arrows by diagram chasing. Hence the induction holds and $H^0(S,\Oo_S(-K_\bbP|_S))^{\otimes n} \to H^n(S,\Oo_S(-nK_\bbP|_S))$ is onto for all $n\geq 1$.

Moreover, we know that K3 surfaces with ADE singularities are projectively normal. Indeed, if $S$ is a singular K3 surface in $\bbP^g$, we may consider a resolution $S' \to S \subset \bbP^g$ with $S'$ a smooth K3 and apply \cite[Theorem 6.1.$(ii)$]{SD}. It follows that the restriction map $H^0(\bbP,\Oo_\bbP(-nK_\bbP)) \to H^0(S,\Oo_S(-nK_\bbP|_S))$ is onto for all $n\geq 1$. Consider the following commutative diagram with $n>1$, whose right column is exact

\begin{center}
\begin{tikzcd}
H^0(S,\Oo_S(-K_\bbP|_S))^{\otimes n} \arrow[r,twoheadrightarrow] & H^0(S,\Oo_S(-nK_\bbP|_S)) \\
H^0(\bbP,\Oo_\bbP(-K_\bbP))^{\otimes n} \arrow[r] \arrow[u,twoheadrightarrow] & H^0(\bbP,\Oo_\bbP(-nK_\bbP)) \arrow[u,twoheadrightarrow] \\
H^0(\bbP,\Oo_\bbP(-K_\bbP))^{\otimes n-1} \arrow[r,twoheadrightarrow] \arrow[u,hookrightarrow] & H^0(\bbP,\Oo_\bbP(-(n-1)K_\bbP)) \arrow[u,hookrightarrow]
\end{tikzcd}
\end{center}

\noindent
We may assume that the bottom arrow is surjective by an induction hypothesis, as it is the case for $n = 1$. By an analogous argument of diagram chasing as above, the middle arrow $H^0(\bbP,\Oo_\bbP(-K_\bbP))^{\otimes n} \to H^n(\bbP,\Oo_\bbP(-nK_\bbP))$ is onto. This is true for all $n$ and thus the conclusion follows from \color{purple}Lemma \ref{lem:very ample characterization}\color{black}.
\end{proof}

\section{Birational models}\hypertarget{section 3}{}

Before constructing the maximal extensions, we study for each $\bbP$ in \color{purple}\hyperlink{Table 2}{Table 2} \color{black}except $\bbP(1,2,3,6)$ a birational model $\varphi : \bbP \dashrightarrow \bbP'$ such that $\varphi$ restricts to an isomorphism on the general $S\in |-K_\bbP|$. This will allow us to express the general $\Gamma\in |-K_\bbP|_S|$ as a complete intersection of two equations of different degrees in $\bbP'$ and to construct extensions of $\Gamma$ in \hyperlink{section 4}{Section 4}.

The first step consists in the introduction of a suitable Veronese map on each $\bbP$, which is an embedding $v_n : \bbP \hookrightarrow X$, where $X$ is a weighted projective space of dimension $4$, so that the anticanonical model $\bbP \to \bbP^{g+1}$ factors as a composition
\begin{center}
\begin{tikzcd}
\bbP \arrow[rr,hookrightarrow,"v_n"] & & X \arrow[rr,dashrightarrow] & & \bbP^{g+1}.
\end{tikzcd}
\end{center}
Each time, $X\dashrightarrow \bbP^{g+1}$ is a rational map that we will specify.

\subsection{The Veronese maps}

Let $\bbP = \bbP(a_0,a_1,a_2,a_3)$ be a weighted projective space. Denote $R = \C[x,y,z,w]$ with the suitable grading such that $\bbP = \mathrm{Proj}(R)$. Then the following isomorphism holds for all $n\in \N$,
$$
\bbP \simeq \mathrm{Proj}(R^{(n)}),
$$
where $R^{(n)}$ is the graded ring whose degree $d$ part is $(R^{(n)})_d = R_{nd}$. This gives rise to an embedding $v_n$ which we refer to as the $n$-Veronese map: given a fixed $n$, $R^{(n)}$ is isomorphic to a quotient of the form $\sfrac{\C[y_0,...,y_m]}{I}$ where $\C[y_0,...,y_m]$ has a given grading $\deg y_i = d_i$ and $I$ is a homogeneous ideal. This makes $\bbP$ a subscheme of a larger weighted projective space, as $\bbP \simeq V(I) = \left\{ \mathrm y = [y_0: \cdots : y_m] \: | \: f(\mathrm y) = 0 \text{ for all } f\in I\right\}$ in $\bbP(d_0,...,d_m)$.

\begin{exemple} \label{ex:5-Veronese}
The $5$-Veronese embedding of $\bbP = \bbP(1,4,5,10)$. Taking $n=5$ yields the isomorphism
$$
\bbP \simeq \mathrm{Proj}(\C[x^5,xy,z,w,y^5]) = \mathrm{Proj}
\left( \sfrac{\C[u_0,u_1,u_2,v,s]}{(u_0s-u_1^5)}\right).
$$
The grading on the right is the following: the $u_i$'s have weight $1$, while $v$ has weight $2$ and $s$ has weight $4$. This realizes $\bbP$ as the degree $5$ hypersurface given by the equation $u_0s= u_1^5$ in $\bbP(1,1,1,2,4)$ through the following embedding.
$$
v_5 : [x:y:z:w] \in \bbP \mapsto [u_0:u_1:u_2:v:s]=[x^5:xy:z:w:y^5].
$$
\end{exemple}

The choice $n=5$ is motivated by the fact that it is a divisor of $\sigma = 20$, with $-K_\bbP = \Oo_\bbP(\sigma)$. We can recover $-K_\bbP$ from $v_5$ using the equality $-K_\bbP = v_5^* \Oo_{\bbP(1,1,1,2,4)}(4)$.

In general, given $\bbP$ any Gorenstein space of the list, we can choose $n$ a divisor of $\sigma$ and embed $\bbP$ by $v_n$ in a larger weighted projective space $X$. This yields $-K_\bbP = v_n^*\Oo_X(\frac{\sigma}{n})$.

For the items on our list from \#9 to \#14 a suitable choice of $n$ is given in \color{purple}\hyperlink{Table 3}{Table 3} \color{black}below, realizing each time $\bbP$ as a hypersurface in a weighted projective space $X$ of dimension $4$. Similarly as in \color{purple}\hyperlink{Table 1}{Table 1}\color{black}, $\sigma$ refers to the degree of $-K_\bbP$ with regard to the grading of $\bbP$.

$$
\hypertarget{Table 3}{}
\begin{array}{|l|l|l|l|l|l|l|}
\hline
\# & \bbP & \sigma & n & X & v_n(\bbP) \subset X & -K_\bbP \\
\hline & & & & & & \\[-10pt]
9 & \bbP(1,4,5,10) & 20 & 5 & \bbP(1,1,1,2,4)_{[u_0:u_1:u_2:v:s]} & \text{quintic } (u_0s=u_1^5) & v_5^*\Oo_X(4) \\
10 & \bbP(1,2,6,9) & 18 & 2 & \bbP(1,1,3,5,9)_{[u_0:u_1:v:s:t]} & 10\text{-ic } (u_0t = s^2) & v_2^*\Oo_X(9) \\
11 & \bbP(1,2,3,6) & 12 & 2 & \bbP(1,1,2,3,3)_{[u_0:u_1:v:s_0:s_1]} & \text{quartic } (u_0s_0=v^2) & v_2^*\Oo_X(6) \\
12 & \bbP(1,3,8,12) & 24 & 3 & \bbP(1,1,3,4,8)_{[u_0:u_1:v:s:t]} & 9\text{-ic } (u_0t=v^3) & v_3^*\Oo_X(8) \\
13 & \bbP(1,6,14,21) & 42 & 7 & \bbP(1,1,2,3,6)_{[u_0:u_1:v:s:t]} & \text{heptic } (u_0t=u_1^7) & v_7^*\Oo_X(6) \\
14 & \bbP(2,3,10,15) & 30 & 3 & \bbP(1,2,4,5,10)_{[u:v:s:t:r]} & 12\text{-ic } (vr=s^3) & v_3^*\Oo_X(10) \\
\hline
\end{array}
$$

\begin{center}
\footnotesize{Table 3}
\end{center}

In each case, the anticanonical embedding $\bbP \hookrightarrow \bbP^{g+1}$ factors as the composite map

\begin{center}
\begin{tikzcd}
\bbP \arrow[r,hookrightarrow,"v_n"] & X \arrow[rr,dashrightarrow,"|\Oo_X(\frac{\sigma}{n})|"] & & \bbP^{g+1}.
\end{tikzcd}
\end{center}

By a divisibility criterion, we can check fairly easily that $|\Oo_X(\frac{\sigma}{n})|$ is not always basepoint-free. This criterion is purely numerical: $|\Oo_X(\frac{\sigma}{n})|$ is basepoint-free if and only if $\frac{\sigma}{n}$ is divisible by all the weights of $X$ (see this \color{purple}\hyperlink{Remark}{Remark}\color{black}, which is based on the same argument). Namely, it is not the case for \#10, \#12 and \#14, for which the induced map $X\dashrightarrow \bbP^{g+1}$ is nonregular.

\subsection{The birational models}

The goal now is to exhibit a birational map from $\bbP$ to another $3$-dimensional weighted projective space $\bbP'$ which realizes the general anticanonical divisor $S$ of $\bbP$ (in other words, a general hyperplane section of the anticanonical model $\bbP \subset \bbP^{g+1}$) as a nongeneral anticanonical divisor of $\bbP'$.

In all cases but \#11, the image of $X$ in $\bbP^{g+1}$ is a cone with vertex a point which belongs to the image of $\bbP \to \bbP^{g+1}$. This follows from the fact that $\frac{\sigma}{n}$ equals the largest weight of $X$; say, $X = \bbP(d_0,d_1,d_2,d_3,d_4)$ with $\frac{\sigma}{n} = d_4$. Then the map given by the linear system $|\Oo_X(\frac{\sigma}{n})|$ is the following:
$$
[x_0:x_1:x_2:x_3:x_4]\in X \mapsto [\mathtt f_0: \mathtt f_1 : \cdots : \mathtt f_g : x_4],
$$
where the $\mathtt f_i$'s form a basis of the degree $\frac{\sigma}{n}$ homogeneous polynomials in the variables $x_0,x_1,x_2$ and $x_3$. Hence the equations for the image of $X$ encode the algebraic relations between the $\mathtt f_i$'s and do not involve $x_4$; the vertex of the cone $X \subset \bbP^{g+1}$ is thus the image of the point $p_{x_4} = [0:0:0:0:1]\in X$, and this point belongs to the image of $\bbP$ by \color{purple}\hyperlink{Table 3}{Table 3}\color{black}, as claimed.

This suggests projecting from the vertex point of this cone to $\bbP^g$. This induces the identity on $S$, provided that $S\subset \bbP^g$ is a hyperplane section of $\bbP$ not passing through the vertex. 

This yields a dominant, generically finite rational map $\varphi : \bbP \dashrightarrow \bbP'$, where $\bbP'$ is the weighted projective $3$-space whose weights are those of $X$ but the last one. In other words, if $X = \bbP(d_0,d_1,d_2,d_3,d_4)$ with $d_4 = \frac{\sigma}{n}$, then $\bbP' = \bbP(d_0,d_1,d_2,d_3)$. This map is in fact birational in each case of \color{purple}\hyperlink{Table 3}{Table 3} \color{black}but item \#11, and we provide a complete proof of its birationality for the particular case \#9 in \color{purple}Lemma \ref{lem:example is birational} \color{black}below. The proof in each other case is similar as the one we will provide for \#9.

The restriction of this map to $S$ is an isomorphism; the image of $S$ is K3 so it is a (nongeneral) anticanonical divisor of $\bbP'$. We denote by $\mathcal L$ the (noncomplete) linear system whose members are the anticanonical divisors of $\bbP'$ which are the direct images of all $D\in |-K_\bbP|$,
$$
\mathcal L = \varphi(|-K_\bbP|) \subset |-K_{\bbP'}|,
$$
so that the surface $\varphi(S)$ is a general member of $\mathcal L$. 

Since $S\simeq \varphi(S)$ we will drop the notation $\varphi(S)$ for the sake of brevity and use $S$ instead. Likewise, we will refer to $\varphi(\Gamma)$ as $\Gamma$. As our computations will show, the restriction of $\mathcal L$ to $S$ has $|\Oo_{\bbP'}(\frac{\sigma}{n})|_S|$ as its mobile part. That way, $\Gamma$ is cut out on $\bbP'$ by two equations of different degrees.

\begin{exemple} \label{ex:birational model}
We have seen in \color{purple}Example \ref{ex:5-Veronese} \color{black}that the $5$-Veronese map on $\bbP=\bbP(1,4,5,10)$ is an embedding $\bbP \hookrightarrow X$ with $X = \bbP(1,1,1,2,4)$. The anticanonical model of $\bbP$ factors via

\begin{center}
\begin{tikzcd}
X \arrow[rr,hookrightarrow,"|\Oo_X(4)|"] & & \bbP^{22}.
\end{tikzcd}
\end{center}

Let $\bbP'=\bbP(1,1,1,2)$. It is embedded in $\bbP^{21}$ by $|\Oo_{\bbP'}(4)|$, making $X$ a cone over $\bbP'$. As $\bbP$ passes through the vertex of $X$, projecting from the vertex point onto $\bbP'$ induces a rational map $\varphi$ from $\bbP$ to $\bbP'$. In homogeneous coordinates, we can express $\varphi$ as
$$
\varphi : [x:y:z:w] \in \bbP \mapsto [u_0:u_1:u_2:v]=[x^5:xy:z:w]\in \bbP'.
$$
This is the $5$-Veronese map without its last component $y^5$.
\end{exemple}

\begin{lemma} \label{lem:example is birational}
The map $\varphi$ given in \color{purple}Example \ref{ex:birational model} \color{black}is birational.
\end{lemma}

\begin{proof}
The map $\varphi$ is toric, i.e., equivariant under the torus actions on $\bbP$ and $\bbP'$. It is encoded by a transformation of the polytopes of $\bbP$ and $\bbP'$.

We prove that $\varphi$ is birational by exhibiting a commutative diagram
\begin{center}
\begin{tikzcd}
 & & \widehat{\bbP} \arrow[dll,"\varepsilon_1",swap] \arrow[drr,"\varepsilon_2"] & & \\
\bbP \arrow[rrrr,dashrightarrow,"\varphi",swap] & & & & \bbP'
\end{tikzcd}
\end{center}
in which $\varepsilon_1$ and $\varepsilon_2$ are weighted blow-ups, and $\widehat{P}$ is a projective toric variety encoded by a fan which involves a subdivision of the fans of $\bbP$ and $\bbP'$.

Let us first construct $\varepsilon_1$. The map $\varphi$ is regular outside the point $p_y = \left\{ x = z = w = 0\right\}$. Since this indeterminacy point is a toric point, i.e., it is fixed by action of the torus on $\bbP$, we may resolve the indeterminacy of $\varphi$ via a toric birational modification given by a weighted blow-up of $\bbP$ at $p_y$. This corresponds to a subdivision of the cone of the affine chart $\left\{ y\neq 0 \right\}$.

We introduce the following elements of the lattice $\Z^3$ in $\R^3$:
$$
\mathtt e_x = \left[ \begin{array}{c}
-4 \\ -5 \\ -10
\end{array} \right], \,
\mathtt e_y = \left[ \begin{array}{c}
1 \\ 0 \\ 0
\end{array} \right], \, 
\mathtt e_z = \left[ \begin{array}{c}
0 \\ 1 \\ 0
\end{array} \right], \,
\mathtt e_w = \left[ \begin{array}{c}
0 \\ 0 \\ 1
\end{array} \right].
$$
The algorithm given in \cite[Proposition 2.8]{RT} gives rise to the following fan in $\Z^3$ for the toric variety $\bbP$.
$$
\Sigma_{\bbP} = \mathrm{Fan}(\mathtt e_x, \mathtt e_y, \mathtt e_z, \mathtt e_w) = \mathrm{Fan}\left( \left[ \begin{array}{c}
-4 \\ -5 \\ -10
\end{array} \right], \left[ \begin{array}{c}
1 \\ 0 \\ 0
\end{array} \right], \left[ \begin{array}{c}
0 \\ 1 \\ 0
\end{array} \right], \left[ \begin{array}{c}
0 \\ 0 \\ 1
\end{array} \right] \right).
$$
This notation means that the cones of the fan $\Sigma_\bbP$ are the ones generated by all possible proper subsets of the family $\left\{ \mathtt e_x,\mathtt e_y, \mathtt e_z, \mathtt e_w \right\}$.

There is a one-to-one correspondence between the cones of $\Sigma_\bbP$ and the toric subsets of $\bbP$ (in other words, the subsets which are stable under the action of the torus). For instance, the toric hypersurface $\left\{ x = 0 \right\}$ is encoded by the cone $\R_+ \mathtt e_x$, and likewise, the toric point $p_x$ is encoded by the cone $\R_+ \mathtt e_y + \R_+ \mathtt e_z + \R_+ \mathtt e_w$.

In particular, the indeterminacy point $p_y$ of $\varphi$ is the origin of the affine chart $\left\{ y \neq 0 \right\}$ and a weighted blow-up of $\bbP$ at this point results in a subdivision of the cone $\R_+ \mathtt e_x + \R_+ \mathtt e_z + \R_+ \mathtt e_w$, in other words, the addition of a new cone of dimension $1$ which is generated by a linear combination over $\N$ of $\mathtt e_x, \mathtt e_z$ and $\mathtt e_w$.

We introduce a new element of the lattice:
$$
\mathtt e_\zeta = \left[ \begin{array}{c}
-1 \\ -1 \\ -2
\end{array} \right].
$$
This element can be obtained as a $\Z$-linear combination of $\mathtt e_x, \mathtt e_z$ and $\mathtt e_w$ as follows:
$$
4\left[ \begin{array}{c}
-1 \\ -1 \\ -2
\end{array} \right] = \left[ \begin{array}{c}
-4 \\ -5 \\ -10
\end{array} \right] + \left[ \begin{array}{c}
0 \\ 1 \\ 0
\end{array} \right] + 2\left[ \begin{array}{c}
0 \\ 0 \\ 1
\end{array} \right].
$$
This gives rise to a variety $\widehat{\bbP}$ which is the weighted blow-up of $\bbP$ at $p_y$ with weights $1,1$ and $2$ whose exceptional divisor is isomorphic to $\bbP(1,1,2)$. This yields a fan $\Sigma_{\widehat{\bbP}}$ for $\widehat{\bbP}$ which contains a subdivision of the cone $\R_+ \mathtt e_x + \R_+ \mathtt e_z + \R_+ \mathtt e_w$ as above. The other cones of maximal dimension in the fan $\Sigma_\bbP$ are left unchanged and are also cones of the fan $\Sigma_{\widehat{\bbP}}$.

We now refer to the construction of homogeneous coordinates on toric varieties which is explained in \cite[\S5.1]{CLS}. This allows us to introduce five toric coordinates on $\widehat{\bbP}$ which we denote by $(\mathtt x, \zeta, \mathtt y, \mathtt z, \mathtt w)$ with the following grading in $\Z^2$, so that $\widehat{\bbP} = \mathrm{Proj}(\C[\mathtt x,\zeta,\mathtt y,\mathtt z,\mathtt w])$.
$$
\begin{array}{c|ccccc}
 & \mathtt x & \zeta & \mathtt y & \mathtt z & \mathtt w \\
\hline
\text{degree} & 1 & 0 & 4 & 5 & 10 \\
\text{in } \Z^2: & 0 & 1 & 1 & 1 & 2
\end{array}
$$
The blow-up map $\varepsilon_1$ from $\widehat{\bbP}$ to $\bbP$ in homogeneous coordinates is
$$
\varepsilon_1 : [\mathtt x:\zeta:\mathtt y:\mathtt z:\mathtt w] \in \widehat \bbP \mapsto [\mathtt x\zeta:\mathtt y\zeta^3:\mathtt z\zeta^4:\mathtt w\zeta^8] \in \bbP.
$$
It is well defined everywhere and contracts the exceptional divisor $\left\{ \zeta = 0\right\}$ to the point $p_y$. Indeed, if we fix a point with representative $(\mathtt x, \zeta, \mathtt y, \mathtt z, \mathtt w)$, $\zeta^{\sfrac{1}{4}}$ a fourth root of $\zeta$ and $\mu = (\zeta^{\sfrac{1}{4}})^{-3}$, then the image of this point in $\bbP$ is 
$$
[\mu \mathtt x \zeta : \mu^4 \mathtt y\zeta^3:\mu^5 \mathtt z \zeta^4:\mu^{10} \mathtt w \zeta^8] =
[\mathtt x\zeta^{\sfrac{1}{4}}:\mathtt y:\mathtt z\zeta^{\sfrac{1}{4}}:\mathtt w\zeta^{\sfrac{1}{2}}].
$$
Let us now construct the weighted blow-up $\varepsilon_2$ and check that $\varphi \circ \varepsilon_1 = \varepsilon_2$. The following is a fan for the weighted projective space $\bbP' = \bbP(1,1,1,2)$.
$$
\Sigma_{\bbP'} = \mathrm{Fan}(\mathtt e_\zeta, \mathtt e_y, \mathtt e_z, \mathtt e_w) = \mathrm{Fan}\left( \left[ \begin{array}{c}
-1 \\ -1 \\ -2
\end{array} \right], \left[ \begin{array}{c}
1 \\ 0 \\ 0
\end{array} \right], \left[ \begin{array}{c}
0 \\ 1 \\ 0
\end{array} \right], \left[ \begin{array}{c}
0 \\ 0 \\ 1
\end{array} \right] \right).
$$
Besides, $\widehat{\bbP}$ is also the weighted blow-up of $\bbP'$ along a toric curve. Indeed, we have
$$
\left[ \begin{array}{c}
-4 \\ -5 \\ -10
\end{array} \right] = 5\left[ \begin{array}{c}
-1 \\ -1 \\ -2
\end{array} \right] + \left[ \begin{array}{c}
1 \\ 0 \\ 0
\end{array} \right],
$$
in other words, $\mathtt e_x = 5\mathtt e_\zeta + \mathtt e_y$, where $\mathtt e_\zeta$ and $\mathtt e_y$ are rays of the fan $\Sigma_{\bbP'}$. This subdivision of the cone $\R_+\mathtt e_\zeta + \R_+\mathtt e_y$ yields the same fan as $\Sigma_{\widehat{\bbP}}$.

The blow-up map $\varepsilon_2$ from $\widehat{\bbP}$ to $\bbP'$ is the following:
$$
\varepsilon_2 : [\mathtt x:\zeta:\mathtt y:\mathtt z:\mathtt w] \in \widehat \bbP \mapsto [\mathtt x^5\zeta:\mathtt x\mathtt y:\mathtt z: \mathtt w]\in \bbP',
$$
contracting the exceptional locus $\left\{ \mathtt x=0\right\}$ to a curve.

Moreover, given any point $[\mathtt x:\zeta:\mathtt y:\mathtt z:\mathtt w] \in \widehat \bbP$, we have
$$
\varphi \circ \varepsilon_1 ([\mathtt x:\zeta:\mathtt y:\mathtt z:\mathtt w]) = [\mathtt x^5\zeta^5:\mathtt x \mathtt y \zeta^4:z\zeta^4:\mathtt w \zeta^8] = \varepsilon_2([\mathtt x:\zeta:\mathtt y:\mathtt z:\mathtt w]).
$$
So $\varphi \circ \varepsilon_1 = \varepsilon_2$, as required.
\end{proof}

As a consequence, this map restricts to an isomorphism on the general hyperplane section $S$ of $\bbP \subset \bbP^{g+1}$, since $\varphi$ is by construction the restriction to $\bbP$ of the projection from the vertex of the cone $X$.

Continuing on \color{purple}Example \ref{ex:birational model}\color{black}, if $S$ is general in the basepoint-free linear system $|-K_\bbP|$, then it avoids the indeterminacy point of $\varphi$. Its image being a K3 surface, it is an anticanonical divisor of $\bbP'$, i.e., a quintic surface in $\bbP(1,1,1,2)$. Using the description of $\varphi$ in homogeneous coordinates, we see that $S$ in $\bbP(1,1,1,2)$ has equation
\begin{equation} \label{eq:equation S}
u_0f_4(\mathbf u,v) + u_1^5 = 0
\end{equation}
with $f_4$ a general homogeneous polynomial of degree $4$ in the variable $\mathbf u = (u_0,u_1,u_2)$ and $v$. Indeed, such an equation pulls back to an equation on $\bbP$ of the form $x^5f_{20}(x,y,z,w) = 0$, where $f_{20}$ is a general $20$-ic on $\bbP$. In other words, the pullback to $\bbP$ of $S\subset \bbP'$ is $S+(x^5)$ where the locus $x = 0$ is contracted by $\varphi$.

Therefore, $\mathcal L = \varphi(|-K_\bbP|) \subset |-K_{\bbP'}|$ consists of those quintic surfaces of the form 
$$
u_0f_4(\mathbf u,v) + \lambda u_1^5 = 0,
$$
with $\deg(f_4) = 4$ and $\lambda \in \C$. The surface $S$ being general in $\mathcal L$, $\lambda$ is non zero and up to scaling, we may assume $\lambda = 1$ as in \color{purple}(\ref{eq:equation S})\color{black}. The base locus of $\mathcal L$ is the curve $\Delta := \left\{ u_0 = u_1 = 0\right\}$. 

\begin{lemma}
In case $\#9$, given a general $S \in |-K_\bbP|$, the general $\Gamma \in |-K_\bbP|_S|$ is cut out on $S$ in $\bbP'$ by a general quartic.
\end{lemma}

\begin{proof}
Let  $S'$ be another general member of $\mathcal L$, i.e., the image under $\varphi$ of a general anticanonical divisor of $\bbP$, which is the zero locus of $u_0f'_4(\mathbf u,v) + u_1^5$ with $f'_4$ a homogeneous quartics, then
$$
S\cap S' = \left\{ u_0f_4 + u_1^5 = u_0f'_4 + u_1^5 = 0 \right\} = S\cap \left\{ u_0(f_4-f'_4) = 0 \right\},
$$
and $f_4-f'_4$ is a general quartic of $\bbP'$. This shows that the restriction $\mathcal L|_S$ has $S\cap \left\{ u_0 = 0 \right\} = \Delta$ as its fixed part, and its mobile part is $|\Oo_{\bbP'}(4)|_S|$. Thus the map from $S$ given by the restriction of $\mathcal L$ is the same map as the one induced by the quartics of $\bbP'$,
$$
S \xrightarrow{|\Oo_{\bbP'}(4)|_S|} \bbP^{21}
$$
so that the curve $\Gamma$ is the pullback to $S$ of a hyperplane of $\bbP^{21}$. Hence, $\Gamma$ is cut out on $S$ by a general quartic of $\bbP'$.
\end{proof}

In conclusion, the curve $\Gamma$ is a complete intersection of degrees $5$ and $4$ given by the two equations
$$
u_0f_4(\mathbf u,v) + u_1^5 = g_4(\mathbf u,v) = 0,
$$
with $g_4$ a general quartic on $\bbP'$. This is summed up in the following commutative diagram
\begin{center}
\begin{tikzcd}
& & \bbP \arrow[rr,dashrightarrow,"\varphi"] \arrow[dll,hookrightarrow,"v_5",swap] \arrow[d,hookrightarrow,"|-K_\bbP|"] & & \bbP' \arrow[d,hookrightarrow,"|\Oo_{\bbP'}(4)|"] \\
X \arrow[rr,hookrightarrow,"|\Oo_X(4)|",swap]& & \bbP^{22} \arrow[rr,dashrightarrow,"\mathrm{pr}"] & & \bbP^{21}
\end{tikzcd}
\end{center}
Here, $\mathrm{pr}$ is the projection map from the vertex point of the cone $X$ onto $\bbP'$.

\vspace{.2cm}
In all cases from \#9 to \#14 but \#11, we can apply similar arguments leading to a description of $\Gamma$ as a complete intersection of two different degrees. All the needed pieces of information are listed in \color{purple}\hyperlink{Table 4}{Table 4} \color{black}and \color{purple}\hyperlink{Table 5}{Table 5} \color{black}below. As in the example above, $\Delta$ refers to the base locus of $\mathcal L$; this is also the fixed part of $\mathcal L|_S$. An expression of $\varphi$ in the coordinates $[x:y:z:w]$ of $\bbP$ is also provided. It makes it possible to check the given equation for $S$ in $\bbP'$. It also makes visible the indeterminacy point of $\varphi$, which we denote by $p$. The proof that $\varphi$ is birational in each case is along the same lines as the proof of \color{purple}Lemma \ref{lem:example is birational}\color{black}.

In the case of $\bbP(1,4,5,10)$, the expression of $\varphi$ is $[x^5:xy:z:w]$. The indeterminacy point $p$ is the point for which $x = z = w = 0$, commonly denoted by $p_y$. In all cases, the indeterminacy point is such a coordinate point, as displayed in \color{purple}\hyperlink{Table 4}{Table 4} \color{black}below.

$$
\hypertarget{Table 4}{}
\begin{array}{|l|l|l|l|l|}
\hline
\# & \bbP & \bbP' & \text{expression of } \varphi & p\\
\hline & & & & \\[-10pt]
9 & \bbP(1,4,5,10) & \bbP(1,1,1,2)_{[u_0:u_1:u_2:v]} & [x^5:xy:z:w] & p_y \\
10 & \bbP(1,2,6,9) & \bbP(1,1,3,5)_{[u_0:u_1:v:s]} & [x^2:y:z:xw] & p_w \\
12 & \bbP(1,3,8,12) & \bbP(1,1,3,4)_{[u_0:u_1:v:s]} & [x^3:y:xz:w] & p_z \\
13 & \bbP(1,6,14,21) & \bbP(1,1,2,3)_{[u_0:u_1:v:s]} & [x^7:xy:z:w] & p_y \\
14 & \bbP(2,3,10,15) & \bbP(1,2,4,5)_{[u:v:s:t]} & [y:x^3:xz:w] & p_z \\
\hline
\end{array}
$$
\begin{center}
\footnotesize{Table 4}
\end{center}

The anticanonical model $\bbP \subset \bbP^{g+1}$ is a hypersurface of the cone $X$ whose vertex point is $p$. The projection map from the point $p$ to $\bbP'$ restrict to an isomorphism on the general $S$ such that $p\notin S$, since $S$ is cut out by a hyperplane.

$$
\hypertarget{Table 5}{}
\begin{array}{|l|l|l|l|}
\hline
\# & \text{equation for } S \text{ in } \bbP' & \Delta & \Gamma \\
\hline & & & \\[-10pt]
9 & u_0f_4(\mathbf u,v) + u_1^5 = 0 & u_0 = u_1 = 0 & S\cap \mathrm{quartic} \\
10 & u_0f_9(\mathbf u,v,s) + s^2 = 0 & u_0 = s = 0 & S\cap 9\mathrm{-ic} \\
12 & u_0f_8(\mathbf u,v,s) + v^3 = 0 & u_0 = v = 0 & S\cap 8\mathrm{-ic} \\
13 & u_0f_6(\mathbf u,v,s) + u_1^7 = 0 & u_0 = u_1 = 0 & S\cap \mathrm{sextic} \\
14 & vf_{10}(u,v,s,t) + s^3 = 0 & v = s = 0 & S\cap 10\mathrm{-ic} \\
\hline
\end{array}
$$
\begin{center}
\footnotesize{Table 5}
\end{center}

We always denote by $f_d$ a general degree $d$ homogeneous polynomial in accordance with the grading of $\bbP'$.

$$
\star \star
$$

The information displayed in Table 5 makes it possible to provide a proof for \color{purple}Theorem \ref{thm:canonical very ample}\color{black}.

\begin{proof}[of Theorem 2.5]
Being a general anticanonical divisor of $\bbP$, $S$ has isolated singularities. By Bertini's Theorem, the general $\Gamma\in |-K_\bbP|_S|$ is smooth.

In case $\#10$, the index of $-K_\bbP|_S$ in $\mathrm{Pic}(S)$ is equal to $3$. Hence $\Gamma = 3C$ where $C$ is a Cartier divisor on $S$. By \color{purple}Lemma \ref{lem:hyperelliptic curve in K3}\color{black}, if $\Gamma$ is hyperelliptic, there exists a line bundle $\mathcal J$ on $S_\mathrm{smooth}$ such that $\mathcal J|_\Gamma$ is a $g_2^1$. Hence $\mathcal J|_C$ has degree $\frac{2}{3}$, which is not possible.

In cases $\#9,\#11$ and $\#12$, the index is equal to $2$, hence $\Gamma = 2C$ for some primitive Cartier divisor $C$ on $S$. Once again, if $\Gamma$ is hyperelliptic, by \color{purple}Lemma \ref{lem:hyperelliptic curve in K3} \color{black}there exists a line bundle $\mathcal J$ on $S_\mathrm{smooth}$ such that $\mathcal J|_\Gamma$ is a $g_2^1$. Therefore, $\mathcal J|_C$ is a $g_1^1$, i.e., $C$ is a rational curve. But we know from the values of $g(C)$ listed in \color{purple}\hyperlink{Table 2}{Table 2} \color{black}that it is not the case.

Lastly, in cases $\#13$ and $\#14$, the index is equal to $1$, so $\Gamma$ is primitive in $\mathrm{Pic}(S)$. We need to use the information which is given in \color{purple}\hyperlink{Table 5}{Table 5} \color{black}and yields:
\begin{enumerate}
\item[$\mathtt I$.] for $\#13$, $\Gamma = -7K_\Sigma$ where $\Sigma$ is a general sextic in $\bbP' = \bbP(1,1,2,3)$. It holds by the adjunction formula that $K_\Gamma = -6K_\Sigma|_\Gamma = \Oo_{\bbP'}(6)|_\Gamma$ which is very ample: indeed, the morphism induced by the linear system $|\mathcal O_{\bbP'}(6)|$ coincides with the $6$-Veronese, which is an embedding $\bbP' \hookrightarrow \bbP^{23}$.
\item[$\mathtt{II}$.] for $\#14$, $\Gamma = -6K_\Sigma$ where $\Sigma$ is a general $10$-ic in $\bbP' = \bbP(1,2,4,5)$. By the adjunction formula, $K_\Gamma = -5K_\Sigma|_\Gamma = \Oo_{\bbP'}(10)|_\Gamma$. But $\Gamma$ does not meet the base point of $\Oo_{\bbP'}(10)$ by the generality assumption. The morphism induced by the linear system $|\mathcal O_{\bbP'}(10)|$ factors as the composition
\begin{center}
\begin{tikzcd}
\bbP(1,2,4,5) \arrow[r,"v_{10}"] & \bbP(1^{17},2) \arrow[rrr,dashrightarrow,"|\mathcal O_{\bbP(1^{17},2)}(1)|"] &&& \bbP^{16},
\end{tikzcd}
\end{center}
where, in a choice of coordinates $[u:v:s:t]$ on $\bbP'$, the $10$-Veronese $v_{10}$ is written
$$
|u:v:s:t] \mapsto [\mathtt f_1 : \cdots : \mathtt f_{17} : s^5] \in \bbP(1^{17},2),
$$
with $\mathtt f_1,...,\mathtt f_{17}$ a basis of the vector space of $10$-ics in the variables $(u,v,s,t)$. The right arrow $\bbP(1^{17},2)\dashrightarrow \bbP^{16}$ is regular outside the unique base point of $\mathcal O_{\bbP'}(10)$, which does not belong to $\Gamma$. By \color{purple}\hyperlink{Table 5}{Table 5} \color{black}above, the curve $\Gamma$ is cut out in $\bbP'$ by a general $10$-ic equation, which we may choose to be $\mathtt f_{17} = 0$ without loss of generality, and the $12$-ic $s^3 = -vf_{10}(u,v,s,t)$ where $\deg f_{10} = 10$. Hence the $10$-Veronese map restricted to $\Gamma$ (which is an isomorphism onto its image) yields 
$$
[u:v:s:t] \in \Gamma \mapsto [\mathtt f_1 : \cdots : \mathtt f_{16} : -s^2vf_{10}(u,v,s,t)] \in \bbP(1^{16},2).
$$
In the above, the $\mathtt f_i$'s with $i = 1,...,16$ form a basis of $H^0(\bbP',\mathcal O_{\bbP'}(10)|_\Gamma)$ and the restriction of the two $10$-ics $s^2v, f_{10}(u,v,s,t)$ to $\Gamma$ can be expressed as a linear combination thereof. In other words, up to the equations which define $\Gamma$, we have:
$$
s^2 v = \sum_{i=1}^{16} \lambda_i \mathtt f_i =: \mathtt l(\lambda_i,\mathtt f_i), \, f_{10}(u,v,s,t) = \sum_{i=1}^{16} \mu_i \mathtt f_i =: \mathtt l(\mu_i,\mathtt f_i).
$$
It follows that the morphism induced by the canonical linear system $|K_\Gamma|$ of $\Gamma$,
$$
[u:v:s:t] \in \Gamma \mapsto [\mathtt f_1 : \cdots : \mathtt f_{16}] \in \bbP^{15},
$$
is an isomorphism onto its image, via the following inverse map
$$
[\mathtt f_1 : \cdots : \mathtt f_{16}] \mapsto [\mathtt f_1 : \cdots : \mathtt f_{16}:\mathtt l(\lambda_i,\mathtt f_i) \mathtt l(\mu_i,\mathtt f_i)]\in \Gamma.
$$
\end{enumerate}
In both cases, $K_\Gamma$ is very ample and so $\Gamma$ is nonhyperelliptic.
\end{proof}

\section{Extensions of $\bbP$}\hypertarget{section 4}{}

Recall from \color{purple}Theorem \ref{thm:canonical curve max extension} \color{black}and the values of $\alpha(\Gamma,K_\Gamma)$ given in \color{purple}\hyperlink{Table 2}{Table 2} \color{black}that $\bbP(1,2,3,6)$ admits no extension; therefore, we focus here on items \#9 to \#14 but \#11 and use the description given for $\Gamma$ as a complete intersection of two different degrees in $\bbP'$ to construct an extension of $\bbP$. In all the following cases except \#14, we manage to construct a maximal extension $Y$ as a hypersurface in a weighted projective space of dimension $2+\alpha(\Gamma,K_\Gamma)$. The last case $\#14$, which is $\bbP = \bbP(2,3,10,15)$, will require additional work.

As a consequence of \color{purple}Theorem \ref{thm:canonical curve max extension}\color{black}, we are assured that if $Y$ contains all the surface extensions of $\Gamma$, then it is the universal extension of $\Gamma$. In cases \#9, \#10 and \#12, we provide a proof that the extension $Y$ of $\bbP$ is the universal extension of its general linear curve section $\Gamma$. This requires the following lemma:

\begin{lemma}
\label{lem:universalcharacterization}
Let $\Gamma \subset \bbP^{g-1}$ be a smooth canonical curve of genus $g\geq 11$ such that $\mathrm{Cliff}(\Gamma)\geq 3$. Let $\alpha = \mathrm{cork}(\Phi_{K_\Gamma})$ and $Y\subset \bbP^{g-1+\alpha}$ be a maximal extension of $\Gamma$ with $\dim Y \geq 2$. Then $Y$ is the universal extension of $\Gamma$ if and only if it contains no cone over $\Gamma$ as a linear surface section.
\end{lemma}

\begin{proof}
Let us denote by $\mathcal H$ the family of $g$-planes $\Lambda \subset \bbP^{g-1+\alpha}$ such that $\Gamma \subset \Lambda$, and $\mathcal S = \bbP(\mathrm{coker}(\Phi_{K_\Gamma}))$ the space which parametrizes the surface extensions of $\Gamma$. Then we have $\mathcal H \simeq \mathcal S \simeq \bbP^{\alpha - 1}$, and the map $\mathcal H \to \mathcal S$ which maps $\Lambda$ to $Y \cap \Lambda$ is well-defined if and only if no linear surface section of $Y$ is a cone over $\Gamma$. In that case, it lifts to a linear map $\C^\alpha \to \C^\alpha$ which has trivial kernel, i.e., it is an isomorphism (see \cite[6.6]{CD'}).

It follows that $Y$ contains all the surface extensions of $\Gamma$ as unique linear sections, if and only if it contains no cone over $\Gamma$ as linear surface section. In this situation, by \color{purple}Theorem \ref{thm:canonical curve max extension}\color{black}, it is the universal extension of $\Gamma$.
\end{proof}

\subsection{$\bbP = \bbP(1,4,5,10)$}

According to \color{purple}\hyperlink{Table 4}{Table 4} \color{black}and \color{purple}\hyperlink{Table 5}{Table 5}\color{black}, the curve $\Gamma$ is cut out on $\bbP' = \bbP(1^3,2)_{[\mathbf u:v]}$ by the equations
$$
u_0f_4(\mathbf u,v) + u_1^5 = g_4(\mathbf u,v) = 0
$$
where $f_4$ and $g_4$ are general homogeneous quartic polynomials. Consider then the equation
$$
u_0s_0 + u_1s_1 + u_2s_2 = u_1^5
$$
where $s_0,s_1$ and $s_2$ are coordinates of weight $4$. This defines a quintic hypersurface $Y$ in $\mathbf{X} = \bbP(1^3,2,4^3)_{[\mathbf u:v:\mathbf s]}$.

\begin{lemma} \label{lem:1 4 5 10 max extension}
The variety $Y$ has a model in $\bbP^{24}$ which is a maximal extension of $\bbP$, i.e., it has dimension $5 = 1+\alpha(\Gamma,K_\Gamma)$ according to \color{purple}\hyperlink{Table 2}{Table 2}\color{black}, contains $\bbP$ as a linear section, and is not a cone.

Moreover, the extension $Y$ contains no cone over $\Gamma$ as a surface linear section.
\end{lemma}

\begin{proof}
Consider the linear system $|\Oo_{\mathbf X}(4)|$. It is very ample, as the morphism it induces coincides with the $4$-Veronese of $\mathbf X$. Hence, by restriction, it realizes $Y$ in $\bbP^{24}$ as a variety of degree $[\Oo_{\mathbf X}(5)] \cdot [\Oo_{\mathbf X}(4)]^5 = \frac{4^5 \times 5}{4^3\times 2} = 40 = 2g - 2$. This model is an extension of $\Gamma$, as
$$
Y \cap \left\{ s_0 = -f_4(\mathbf u,v),\, s_1 = s_2 = g_4(\mathbf u,v) = 0 \right\} = \Gamma.
$$
Moreover, we have by adjunction $K_\Gamma = \mathcal O_\mathbf{X}(4)|_\Gamma$, so the linear section $\Gamma \subset Y$ is the canonical model of $\Gamma$, as required.

The fivefold $Y$ in $\bbP^{24}$ is a maximal extension of $\Gamma$ by \color{purple}Theorem \ref{thm:canonical curve max extension} \color{black}since it has dimension $1+\alpha(\Gamma,K_\Gamma)$ and it contains $\bbP = \bbP(1,4,5,10)$ as a $3$-fold linear section. Indeed, as indicated in \color{purple}\hyperlink{Table 3}{Table 3}\color{black}, $\bbP$ embeds into $\bbP(1,1,1,2,4)_{[\mathbf u:v:s_0]}$ as the quintic hypersurface $u_0s_0 = u_1^5$. It follows that $Y$ is an extension of $\bbP$, as 
$$
\bbP = Y\cap \left\{ s_1 = s_2 = 0 \right\}.
$$
Now we assume by contradiction that $Y$ is a cone in $\bbP^{24}$. Then it admits a linear surface section which is a cone over $\Gamma$. But such a cone in $Y$ does not exist, by the following argument.

We identify the total family of surface extensions of $\Gamma$ inside $Y$ via the following method: the equations which cut out $\Gamma$ from $Y$ being
$$
s_0 + f_4(\mathbf u,v) = s_1 = s_2 = g_4(\mathbf u,v) = 0,
$$
any surface extension of $\Gamma$ which is a linear section of $Y\subset \bbP^{24}$ is given by three linear combinations of the equations above. In other words, any surface extension of $\Gamma$ inside $Y$ is given by a set of four equations in $\mathbf X = \bbP(1^3,2,4^3)_{[\mathbf u:v:\mathbf s]}$ as follows:
\begin{equation}
\label{eq:firsthyperplane}
\lambda_0(s_0 + f_4(\mathbf u,v)) + \lambda_1 s_1 + \lambda_2 s_2 + \lambda_3 g_4(\mathbf u,v) = 0,
\end{equation}
\begin{equation}
\label{eq:secondhyperplane}
\mu_0(s_0 + f_4(\mathbf u,v)) + \mu_1 s_1 + \mu_2 s_2 + \mu_3 g_4(\mathbf u,v) = 0,
\end{equation}
\begin{equation}
\label{eq:thirdhyperplane}
\eta_0(s_0 + f_4(\mathbf u,v)) + \eta_1 s_1 + \eta_2 s_2 + \eta_3 g_4(\mathbf u,v) = 0,
\end{equation}
\begin{equation}
\label{eq:Yequation}
u_0s_0 + u_1s_1 + u_2s_2 = u_1^5,
\end{equation}
where the three vectors $(\lambda_0,\lambda_1,\lambda_2,\lambda_3), (\mu_0,\mu_1,\mu_2,\mu_3)$ and $(\eta_0,\eta_1,\eta_2,\eta_3)$ are linearly independent.

Assume by contradiction that there exists such a linear surface section $S$ which is a cone over $\Gamma$, given by the four equations \color{purple}(\ref{eq:firsthyperplane})\color{black}, \color{purple}(\ref{eq:secondhyperplane})\color{black}, \color{purple}(\ref{eq:thirdhyperplane}) \color{black}and \color{purple}(\ref{eq:Yequation})\color{black}. In particular, $S$ contains a line $\ell \subset \mathbf X$ such that $\ell \cdot \mathcal O_\mathbf{X}(4) = 1$ and $\ell \cdot \mathcal O_\mathbf{X}(2) = \frac{1}{2}$. Let $U \subset \mathbf X$ denote the locus where $\mathcal O_\mathbf{X}(2)$ is locally free: it is the complement in $\mathbf 
X$ of the vanishing locus
$$
u_0 = u_1 = u_2 = v = 0,
$$
i.e., $U$ is the complement of the weighted projective subspace spanned by the three points $p_{s_0}$, $p_{s_1}$ and $p_{s_2}$.

\begin{enumerate}
\item[$\bullet$] If there exists a line $\ell \subset S$ such that $\ell \subset U$, then $\mathcal O_\mathbf{X}(2)$ is Cartier on $\ell$, which yields a contradiction with  $\ell \cdot \mathcal O_\mathbf{X}(2) = \frac{1}{2}$.
\item[$\bullet$] If all the lines $\ell \subset S$ meet $\mathbf X - U$, then the vertex point of the cone $S$ belongs to $\mathbf X - U$. Such a point is of the form
$$
[0:0:0:0:s_0:s_1:s_2]
$$
with respect to the coordinates $[u_0:u_1:u_2:v:s_0:s_1:s_2]$ on $\mathbf X$. Up to a change of the coordinates $s_0,s_1,s_2$, we may assume that the vertex point is 
$$
p_{s_2} = [0:0:0:0:0:0:1].
$$ 
In other words, there exists a change of variables of the form 
$$
[s_0:s_1:s_2] \mapsto A\cdot [s_0:s_1:s_2],
$$
with 
$$
A = \left( \begin{array}{ccc}
a_{00} & a_{01} & a_{02} \\
a_{10} & a_{11} & a_{12} \\
a_{20} & a_{21} & a_{22}
\end{array} \right) \in \mathrm{GL}_3(\C),
$$
which moves the vertex point of $S$ to $p_{s_2}$. Such a linear transformation of the coordinates $s_i$ applied to the equations \color{purple}(\ref{eq:firsthyperplane})\color{black}, \color{purple}(\ref{eq:secondhyperplane})\color{black}, \color{purple}(\ref{eq:thirdhyperplane}) \color{black}and \color{purple}(\ref{eq:Yequation}) \color{black}yields
\begin{equation}
\label{eq:firsthyperplane'}
\begin{array}{rccl}
0 & = & & \lambda_0(a_{00}s_0+a_{01}s_1+a_{02}s_2) \\
& & + & \lambda_1 (a_{10}s_0+a_{11}s_1+a_{12}s_2) \\
& & + & \lambda_2 (a_{20}s_0+a_{21}s_1+a_{22}s_2) \\
& & + & \lambda_0f_4 + \lambda_3 g_4 
\end{array}
\end{equation}
\begin{equation}
\label{eq:secondhyperplane'}
\begin{array}{rccl}
0 & = & & \mu_0(a_{00}s_0+a_{01}s_1+a_{02}s_2) \\
& & + & \mu_1 (a_{10}s_0+a_{11}s_1+a_{12}s_2) \\
& & + & \mu_2 (a_{20}s_0+a_{21}s_1+a_{22}s_2) \\
& & + & \mu_0 f_4 + \mu_3 g_4
\end{array}
\end{equation}
\begin{equation}
\label{eq:thirdhyperplane'}
\begin{array}{rccl}
0 & = & & \eta_0(a_{00}s_0+a_{01}s_1+a_{02}s_2) \\
& & + & \eta_1 (a_{10}s_0+a_{11}s_1+a_{12}s_2) \\
& & + & \eta_2 (a_{20}s_0+a_{21}s_1+a_{22}s_2) \\
& & + & \eta_0 f_4 + \eta_3 g_4
\end{array}
\end{equation}
\begin{equation}
\label{eq:Yequation'}
\begin{array}{rccl}
0 & = & & u_0(a_{00}s_0+a_{01}s_1+a_{02}s_2) \\
 & & + & u_1(a_{10}s_0+a_{11}s_1+a_{12}s_2) \\
 & & + & u_2(a_{20}s_0+a_{21}s_1+a_{22}s_2) \\
 & & - & u_1^5.
\end{array}
\end{equation}
Since $S|_{s_2 = 1}$ is an affine cone, the ideal $I(S)$ of $S$ in the ring $\C[u_0,u_1,u_2,v,s_0,s_1,s_2]$ admits four generators (namely, three of degree $4$ and one of degree $5$) whose restrictions to $s_2 = 1$ are homogeneous. Note that in the above, the equation \color{purple}(\ref{eq:Yequation'}) \color{black} has degree $5$, while \color{purple}(\ref{eq:firsthyperplane'})\color{black}, \color{purple}(\ref{eq:secondhyperplane'})\color{black}, \color{purple}(\ref{eq:thirdhyperplane'}) \color{black}have degree $4$. As a consequence, as generators of the ideal $I(S)$ not involving the variable $s_2$, there are 
\begin{enumerate}
\item[$(i)$] three homogeneous quartics which are linear combinations of \color{purple}(\ref{eq:firsthyperplane'})\color{black}, \color{purple}(\ref{eq:secondhyperplane'}) \color{black}and \color{purple}(\ref{eq:thirdhyperplane'}) \color{black}with constant coefficients, 
\item[$(ii)$] one homogeneous quintic which is a combination of \color{purple}(\ref{eq:Yequation'}) \color{black}with \color{purple}(\ref{eq:firsthyperplane'})\color{black}, \color{purple}(\ref{eq:secondhyperplane'}) \color{black}and \color{purple}(\ref{eq:thirdhyperplane'})\color{black}.
\end{enumerate}
Taking a linear combination as $(i)$ suggests, up to a change of notation for coefficients $\lambda_i,\mu_i,\eta_i$, we obtain three equations of the same form as \color{purple}(\ref{eq:firsthyperplane'})\color{black}, \color{purple}(\ref{eq:secondhyperplane'}) \color{black}and \color{purple}(\ref{eq:thirdhyperplane'}) \color{black}which do not involve $s_2$.

Then the condition $(ii)$ implies $a_{02}=a_{12}=a_{22} = 0$, which is not possible.
\end{enumerate}
\end{proof}

Letting $(\lambda_0,\lambda_1,\lambda_2)$ move in $\C^3$, we get a family of K3 surfaces 
$$
Y\cap \left\{ s_0 = \lambda_0g_4(\mathbf u,v) - f_4(\mathbf u,v),\, s_1 = \lambda_1g_4(\mathbf u,v),\, s_2 = \lambda_2g_4(\mathbf u,v) \right\}
$$
which are all linear sections of $Y$ and contain $\Gamma$ as a hyperplane section. Indeed, the curve $\Gamma$ is cut out on all of them by $\left\{ g_4(\mathbf u,v) = 0\right\}$. Among them, those that are members of the linear system $\mathcal L$ are those parameterized by $\lambda_1 = \lambda_2 = 0$. The other surface extensions of $\Gamma$ could not be recovered from $\bbP$ and require the introduction of the birational model $\varphi : \bbP \dashrightarrow \bbP'$.

\begin{lemma} \label{lem:1 4 5 10 universal extension}
The variety $Y$ in $\bbP^{24}$ is the universal extension of $\Gamma$.
\end{lemma}

\begin{proof}
By \color{purple}Lemma \ref{lem:universalcharacterization}\color{black}, we may conclude that $Y$ is the universal extension of $\Gamma$, since it contains no cone over $\Gamma$ as a linear surface section.
\end{proof}

As a sanity check, let us see what happens if we try to build a larger extension of $\bbP$. One might consider adding a coordinate $s_3$ of weight $4$ and the hypersurface
$$
\ell_0s_0 + \ell_1s_1 + \ell_2s_2 + \ell_3s_3 = u_1^5
$$
where $\ell_0,\ell_1,\ell_2,\ell_3$ are degree $1$ homogeneous polynomials in the variables $u_0,u_1$ and $u_2$. This is $6$-dimensional and indeed contains $\Gamma$ as a linear section; however, the $\ell_i$'s are linearly dependent and thus the variety given by the equation above is a cone over $Y$.

\subsection{$\bbP = \bbP(1,2,6,9)$}

By \color{purple}\hyperlink{Table 4}{Table 4} \color{black}and \color{purple}\hyperlink{Table 5}{Table 5}\color{black}, the curve $\Gamma$ is given in $\bbP(1,1,3,5)_{[\mathbf u:v:s]}$ by the following two equations
$$
u_0f_9(\mathbf u,v,s) + s^2 = g_9(\mathbf u,v,s) = 0
$$
where $f_9$ and $g_9$ are general homogeneous polynomials of degree $9$. Adding two coordinates $t_0$ and $t_1$ of weight $9$, consider the $10$-ic hypersurface $Y$ in $\mathbf X = \bbP(1^2,3,5,9^2)_{[\mathbf u:v:s:\mathbf t]}$ given by the equation
$$
u_0t_0 + u_1t_1 = s^2.
$$

\begin{lemma}
\label{lem:1 2 6 9 extension very ample}
The restriction of $\mathcal O_\mathbf X(9)$ to $Y$ is a very ample line bundle.
\end{lemma}

\begin{proof}
The linear system $|\Oo_{\mathbf X}(9)|$ has one base point in $\mathbf X$ but its restriction to $Y$ is basepoint-free, hence a line bundle. The rational map $\omega : \mathbf X \dashrightarrow \bbP^{30}$ induced by the linear system $|\mathcal O_\mathbf X(9)|$ factorizes via the $9$-Veronese of $\mathbf X$, as the following composition:
\begin{center}
\begin{tikzcd}
\mathbf X \arrow[r,hookrightarrow,"v_9"] & \bbP(1^{31},2^5,3^3,4^2,5) \arrow[rr,dashrightarrow,"|\mathcal O(1)|"] & & \bbP^{30}.
\end{tikzcd}
\end{center}
The $9$-Veronese embedding $v_9$ is written as follows in the coordinates $[u_0:u_1:v:s:t_0:t_1]$ of $\mathbf X$:
$$
[\mathbf u:v:s:\mathbf t] \mapsto [\underline{\mathtt f}:s^3\underline{\mathbf h^{(3)}}:s^5\underline{\mathbf h^{(2)}}:s^7u_0:s^7u_1:s^9]\in \bbP(1^{31},2^5,3^3,4^2,5),
$$
with the notations
$$
s^5\underline{\mathbf h^{(2)}} = [s^5u_0^2:s^5u_0u_1:s^5u_1^2],
$$
$$
s^3\underline{\mathbf h^{(3)}} = [s^3v:s^3u_0^3:s^3u_0^2u_1:s^3u_0u_1^2:s^3u_1^3],
$$
$$
\underline{\mathtt f} = [\mathtt f_0: \mathtt f_1: \cdots : \mathtt f_{30}] \text{ a base of 9-ics on } \mathbf X,
$$
so that, without loss of generality, we may assume $\mathtt f_0 = t_0, \mathtt f_1 = t_1$ and $\mathtt f_i$ involves only the variables $u_0,u_1,v$ and $s$ as soon as $i\geq 2$. The hypersurface $Y$ of $\mathbf X$ being given by the equation $s^2 = u_0t_0+u_1t_1$, the restriction of $v_9$ to $Y$, which is an isomorphism onto its image, is given by the expression above with
$$
s^9 = s(u_0t_0+u_1t_1)^4,
$$
$$
s^7u_i = s(u_0t_0+u_1t_1)^3u_i \text{ for } i=0,1,
$$
$$
s^5\underline{\mathbf h^{(2)}} = s(u_0t_0+u_1t_1)^2\underline{\mathbf h^{(2)}},
$$
$$
s^3\underline{\mathbf h^{(3)}} = s(u_0t_0+u_1t_1)\underline{\mathbf h^{(3)}}.
$$
To prove that $\mathcal O_\mathbf X(9)|_Y$ is very ample, it is enough to prove that $v_9(Y)$ can be recovered from $\omega(Y)$, i.e., from the $\mathtt f_i$'s. It is the case, as
$$
s(u_0t_0+u_1t_1)^4 = s\sum_{i=0}^4 {4\choose i} u_0^it_0^i + u_1^{4-i}t_1^{4-i},
$$
in which we have assumed $t_0 = \mathtt f_0,t_1 = \mathtt f_1$, and we can express each $su_0^iu_1^{4-i}$ as a linear combination of the other $\mathtt f_i$'s since it has degree $9$; moreover,
$$
s(u_0t_0+u_1t_1)^3u_i = su_i\sum_{j=0}^3 {3\choose j} u_0^jt_0^j + u_1^{3-j}t_1^{3-j},
$$
where each $su_iu_0^{j}u_1^{3-j}$ is a $9$-ic, hence is also a linear combination of the $\mathtt f_i$'s; furthermore,
$$
s(u_0t_0+u_1t_1)^2\underline{\mathbf h^{(2)}} = s\underline{\mathbf h^{(2)}}(u_0^2t_0^2 + 2u_0u_1t_0t_1 + u_1^2t_1^2),
$$
where $s\underline{\mathbf h^{(2)}}u_0^2,s\underline{\mathbf h^{(2)}}u_0u_1$ and $s\underline{\mathbf h^{(2)}}$ are collections of $9$-ics, hence of linear combinations of the $\mathtt f_i$'s; lastly, in
$$
s(u_0t_0 + u_1t_1)\underline{\mathbf{h^{(3)}}},
$$
both $su_0\underline{\mathbf h^{(3)}}$ and $su_1\underline{\mathbf h^{(3)}}$ are collections of $9$-ics, hence of linear combinations of the $\mathtt f_i$'s. We may thus conclude that $v_9(Y)$ can be expressed in coordinates using only the $\mathtt f_i$'s, meaning that there exists an inverse morphism $\omega(Y) \to v_9(Y) \simeq Y$ to $\omega|_Y$.
\end{proof}

\begin{lemma}
The variety $Y$ has a model in $\bbP^{30}$ which is a maximal extension of $\bbP$.
\end{lemma}

\begin{proof}
By \color{purple}Lemma \ref{lem:1 2 6 9 extension very ample}\color{black}, the linear system $|\mathcal O_\mathbf X(9)|_Y|$ defines an embedding which realizes $Y$ as a projective variety in $\bbP^{30}$ of degree $\frac{9^4\times 10}{9^2\times 5 \times 3} = 54 = 2g-2$. It has dimension $4 = 1+\alpha(\Gamma,K_\Gamma)$ by \color{purple}\hyperlink{Table 2}{Table 2} \color{black}and contains $\Gamma$ as a linear section in $\bbP^{27}$:
$$
Y\cap \left\{ t_0 = -f_9(\mathbf u,v,s),\, t_1 = g_9(\mathbf u,v,s) = 0 \right\} = \Gamma.
$$
Besides, the embedding of $Y$ in $\bbP^{30}$ is not a cone, by the same arguments as those mentioned in the proof of \color{purple}Lemma \ref{lem:1 4 5 10 max extension}\color{black}. The fourfold $Y$ is thus a maximal extension of $\Gamma$. Recall from \color{purple}\hyperlink{Table 3}{Table 3} \color{black}that $\bbP$ is the hypersurface of $\bbP(1,1,3,5,9)_{[\mathbf u:v:s:t_0]}$ given by the equation $u_0t_0 = s^2$. This shows that $\bbP = Y\cap \left\{ t_1=0 \right\}$.
\end{proof}

In particular,
$$
Y \cap \left\{ t_0 = \lambda_0 g_9(\mathbf u,v,s) - f_9(\mathbf u,v,s),\, t_1 = \lambda_1g_9(\mathbf u,v,s) \right\}
$$
describes a family of K3 surfaces in $Y$ indexed by $(\lambda_0,\lambda_1) \in \C^2$ which all contain $\Gamma$ as a hyperplane section. Among them, the members of $\mathcal L$, i.e., anticanonical divisors of $\bbP$, are the ones for which $\lambda_1 = 0$.

\begin{lemma}
The variety $Y$ in $\bbP^{30}$ is the universal extension of $\Gamma$.
\end{lemma}

The proof of this lemma follows a similar argument as the proof of \color{purple}Lemma \ref{lem:1 4 5 10 universal extension}\color{black}. From the equations of $\Gamma$ in $\mathbf X = \bbP(1^2,3,5,9^2)_{[\mathbf u:v:s:\mathbf t]}$:
$$
u_0t_0 + u_1t_1 - s^2 = t_0 + f_9(\mathbf u,v,s) = t_1 = g_9(\mathbf u,v,s) = 0,
$$
the first one of which being the equation of $Y$ in $\mathbf X$, we identify the whole family of surface extensions of $\Gamma$ in $Y$. This family is parametrized by $\bbP^2$ and it consists of all the intersections of $Y$ with the vanishing locus of any two independent linear combinations of
$$
t_0 + f_9(\mathbf u,v,s),\, t_1 \text{ and } g_9(\mathbf u,v,s).
$$
One shows by the same argument as in \color{purple}Lemma \ref{lem:1 4 5 10 max extension} \color{black}that this family contains no cone, and we may conclude that $Y$ is the universal extension of $\Gamma$ by \color{purple}Lemma \ref{lem:universalcharacterization}\color{black}.

\subsection{$\bbP = \bbP(1,3,8,12)$}

By \color{purple}\hyperlink{Table 4}{Table 4} \color{black}and \color{purple}\hyperlink{Table 5}{Table 5}\color{black}, the curve $\Gamma$ is given in $\bbP(1,1,3,4)_{[\mathbf u:v:s]}$ by the following equations
$$
u_0f_8(\mathbf u,v,s) + v^3 = g_8(\mathbf u,v,s) = 0,
$$
where $f_8$ and $g_8$ are general homogeneous polynomials of degree $8$. After adding two coordinates $t_0$ and $t_1$ of weight $8$, consider the $9$-ic hypersurface $Y$ in $\mathbf X = \bbP(1^2,3,4,8^2)_{[\mathbf u:v:s:\mathbf t]}$ of equation
$$
u_0t_0 + u_1t_1 = v^3.
$$

\begin{lemma}
The variety $Y$ has a model in $\bbP^{27}$ which is a maximal extension of $\bbP$.
\end{lemma}

\begin{proof}
It is embedded in $\bbP^{27}$ by the restriction of the linear system $|\Oo_{\mathbf X}(8)|$; the very ampleness of $\mathcal O_\mathbf X(8)|_Y$ being proven by a similar argument as \color{purple}Lemma \ref{lem:1 2 6 9 extension very ample}\color{black}. This model has degree $\frac{8^4\times 9}{8^2\times 4 \times 3} = 46 = 2g-2$ and dimension $4 = 1+\alpha(\Gamma,K_\Gamma)$ by \color{purple}\hyperlink{Table 2}{Table 2} \color{black}and contains $\Gamma$ as a linear section:
$$
Y\cap \left\{ t_0 = -f_8(\mathbf u,v,s),\, t_1 = g_8(\mathbf u,v,s) = 0 \right\} = \Gamma.
$$
Hence it is a maximal extension of $\Gamma$. It is also an extension of $\bbP$; indeed, we know from \color{purple}\hyperlink{Table 3}{Table 3} \color{black}that $\bbP$ is the hypersurface $u_0t_0 = v^3$ is $\bbP(1,1,3,4,8)_{[\mathbf u:v:s:t_0]}$. This exhibits $\bbP$ as a hyperplane section of $Y$, which is $\bbP = Y\cap \left\{ t_1 = 0 \right\}$. The fact that the image of $Y$ in $\bbP^{27}$ is not a cone can be proven in the same way as in \color{purple}Lemma \ref{lem:1 4 5 10 max extension}\color{black}.
\end{proof}

Letting $(\lambda_0,\lambda_1)$ move in $\C^2$, we get a family 
$$
Y\cap \left\{ t_0 = \lambda_0 g_8(\mathbf u,v,s) - f_8(\mathbf u,v,s),t_1 = \lambda_1 g_8(\mathbf u,v,s) \right\}
$$
of K3 surfaces in $Y$ which contain $\Gamma$ as a hyperplane section. The surfaces in this family that are members of $\mathcal L$, i.e., anticanonical divisors of $\bbP$, are the ones for which $\lambda_1 = 0$.

\begin{lemma}
The variety $Y$ in $\bbP^{27}$ is the universal extension of $\Gamma$.
\end{lemma}

The proof of this lemma follows a similar argument as the proof of \color{purple}Lemma \ref{lem:1 4 5 10 universal extension}\color{black}. From the equations of $\Gamma$ in $\mathbf X = \bbP(1^2,3,4,8^2)_{[\mathbf u:v:s:\mathbf t]}$:
$$
u_0t_0 + u_1t_1 - v^3 = t_0 + f_8(\mathbf u,v,s) = t_1 = g_8(\mathbf u,v,s) = 0,
$$
the first one of which being the equation of $Y$ in $\mathbf X$, we identify the whole family of surface extensions of $\Gamma$ in $Y$. This family is parametrized by $\bbP^2$ and it consists of all the intersections of $Y$ with the vanishing locus of any two independent linear combinations of
$$
t_0 + f_8(\mathbf u,v,s),\, t_1 \text{ and } g_8(\mathbf u,v,s).
$$
One shows by the same argument as \color{purple}Lemma \ref{lem:1 4 5 10 max extension} \color{black}that this family contains no cone, and we may conclude that $Y$ is the universal extension of $\Gamma$ by \color{purple}Lemma \ref{lem:universalcharacterization}\color{black}.

\subsection{$\bbP = \bbP(1,6,14,21)$}

By \color{purple}\hyperlink{Table 4}{Table 4} \color{black}and \color{purple}\hyperlink{Table 5}{Table 5}\color{black}, the curve $\Gamma$ is given in $\bbP(1,1,2,3)_{[\mathbf u:v:s]}$ by the following equations
$$
u_0f_6(\mathbf u,v,s) + u_1^7 = g_6(\mathbf u,v,s) = 0,
$$
with $f_6$ and $g_6$ general homogeneous polynomials of degree $6$. Adding two coordinates  $t_0$ and $t_1$ of weight $6$, consider the heptic hypersurface $Y$ in $\mathbf X = \bbP(1^2,2,3,6^2)_{[\mathbf u:v:s:\mathbf t]}$ given by
$$
u_0t_0 + u_1t_1 = u_1^7.
$$

\begin{lemma}
The variety $Y$ has a model in $\bbP^{24}$ which is a maximal extension of $\bbP$.
\end{lemma}

\begin{proof}
The line bundle $\mathcal O_{\mathbf X}(6)$ is very ample, since the map induced by the linear system $|\mathcal O_{\mathbf X}(6)|$ coincides with the $6$-Veronese embedding of $\mathbf X$. This model realizes $Y$ as a variety of degree $\frac{6^4\times 7}{6^2\times 3\times 2} = 42 = 2g-2$ in $\bbP^{24}$, which contains $\Gamma$ as a linear section:
$$
Y\cap \left\{ t_0 = -f_6(\mathbf u,v,s),\, t_1 = g_6(\mathbf u,v,s) = 0 \right\} = \Gamma.
$$
By \color{purple}\hyperlink{Table 2}{Table 2}\color{black}, $Y$ has dimension $4 = 1+\alpha(\Gamma,K_\Gamma)$. Hence, it is a maximal extension of $\Gamma$. It is also an extension of $\bbP$: recall from \color{purple}\hyperlink{Table 3}{Table 3} \color{black}that $\bbP$ is the hypersurface $u_0t_0 = u_1^7$ in $\bbP(1,1,2,3,6)_{[\mathbf u:v:s:t_0]}$. As a consequence, we have the equality $Y\cap \left\{ t_1 = 0 \right\} = \bbP$.

Now we assume by contradiction that $Y$ is a cone over $\bbP$. By the equality above, we may assume up to a change of variables that the vertex of this cone is the point $p_{t_1} = \left\{ \mathbf u = v = s = t_0 = 0 \right\}$. In other words, there must exist a change of variables on $\mathbf X = \bbP(1^2,2,3,6^2)_{[\mathbf u:v:s:\mathbf t]}$ of the form
$$
[\mathbf u:v:s:\mathbf t] \mapsto [\mathbf u:v:s:at_0+bt_1 + \alpha(\mathbf u,v,s):ct_0+dt_1 + \beta(\mathbf u,v,s)],
$$
where $\alpha$ and $\beta$ are homogeneous sextics in $(\mathbf u,v,s)$ and $ad \neq bc$, which eliminates the coordinate $t_1$ from the equation of $Y$: 
$$
u_0t_0 + u_1t_1 = u_1^7,
$$
i.e., makes the affine chart $Y|_{t_1 = 1}$ an affine cone. But such a change of variables does not exist. We may thus conclude that $Y$ is a maximal extension of $\bbP$ and $\Gamma$.
\end{proof}

Letting $(\lambda_0,\lambda_1)$ move in $\C^2$, we have a family 
$$
Y\cap \left\{ t_0 = \lambda_0 g_6(\mathbf u,v,s) - f_6(\mathbf u,v,s),\, t_1 = \lambda_1 g_6(\mathbf u,v,s) \right\}
$$
of K3 surfaces in $Y$ which are extensions of $\Gamma$. Those surfaces which are members of $\mathcal L$, i.e., anticanonical divisors of $\bbP$, are the ones for which $\lambda_1 = 0$.

The question arises whether there exists a cone over $\Gamma$ in $\bbP^{22}$ as a linear section of $Y$. If the answer is no, then $Y$ is the universal extension of $\Gamma$ by \color{purple}Lemma \ref{lem:universalcharacterization}\color{black}. However, the argument used in \color{purple}Lemma \ref{lem:1 4 5 10 max extension} \color{black}doesn't apply here, since a curve $\ell\subset S_\lambda$ such that $\ell\cdot [\Oo_{\bbP'}(6)] = 1$ could pass through the base points of $\Oo_{\bbP'}(2)$ and $\Oo_{\bbP'}(3)$, allowing $\ell\cdot [\Oo_{\bbP'}(2)] = \frac{1}{3}$ and $\ell\cdot [\Oo_{\bbP'}(3)] = \frac{1}{2}$.

\subsection{$\bbP = \bbP(2,3,10,15)$}

By \color{purple}\hyperlink{Table 4}{Table 4} \color{black}and \color{purple}\hyperlink{Table 5}{Table 5}\color{black}, the curve $\Gamma$ is given in $\bbP(1,2,4,5)_{[u:v:s:t]}$ by the equations
$$
vf_{10}(u,v,s,t) + s^3 = g_{10}(u,v,s,t) = 0,
$$
with $f_{10}$ and $g_{10}$ general homogeneous polynomials of degree $10$. After adding two coordinates $r_0$ and $r_1$ of weight $10$ we construct an extension of $\Gamma$ as the hypersurface 
$$
u^2r_0 + vr_1 = s^3
$$
which we denote by $Y_1$, in $\mathbf X = \bbP(1,2,4,5,10^2)_{[u:v:s:t:\mathbf r]}$.

\begin{lemma} \label{lem:2 3 10 15 nonmaximal extension 1}
The variety $Y_1$ has a model in $\bbP^{18}$ which is an extension of $\bbP$. However, it is not a maximal extension of $\Gamma$.
\end{lemma}

\begin{proof}
It is embedded in $\bbP^{18}$ by the restriction of the linear system $|\Oo_{\mathbf X}(10)|$; the very ampleness of $\mathcal O_\mathbf X(10)|_{Y_1}$ being proven by a similar argument as \color{purple}Lemma \ref{lem:1 2 6 9 extension very ample}\color{black}. This model has degree $\frac{10^4\times 12}{10^2\times 5\times 4\times 2} = 30 = 2g-2$ and contains $\Gamma$ as a linear section:
$$
Y_1 \cap \left\{ r_1 = -f_{10}(u,v,s,t), r_0 = g_{10}(u,v,s,t) = 0 \right\} = \Gamma.
$$
In accordance with the equation $vr_1 = s^3$ which is given for $\bbP$ in \color{purple}\hyperlink{Table 3}{Table 3} \color{black}as a hypersurface in $\bbP(1,2,4,5,10)_{[u:v:s:t:r_1]}$, one checks that $Y_1\cap \left\{ r_0 = 0\right\} = \bbP$.
However, $Y_1$ has dimension $4$, while $1+\alpha(\Gamma,K_\Gamma) = 5$, so the extension isn't maximal.
\end{proof}

The two next subsections are devoted to the construction of a maximal extension. The strategy is to introduce another birational model for $\bbP$ to construct another nonmaximal extension $Y_2$. The data of $Y_1$ and $Y_2$ will allow us to construct in \color{purple}\hyperlink{§4.5.2}{§4.5.2} \color{black}a maximal extension of $\Gamma$ and $\bbP$.

\subsubsection{An alternative model for $\bbP = \bbP(2,3,10,15)$}

Here, we construct $Y_2$. Introducing homogeneous coordinates $[u':v':s':t']$ on the weighted projective space $\bbP(1,3,5,9)$, consider the following rational map $\psi$ from $\bbP$ to $\bbP(1,3,5,9)$
$$
\psi : [x:y:z:w]\in \bbP \mapsto [u':v':s':t'] = [x:y^2:z:yw]\in \bbP(1,3,5,9).
$$
The expression of $\psi$ in homogeneous coordinates is obtained from the $2$-Veronese map $v_2$ on $\bbP$,
$$
[x:y:z:w] \mapsto [x:y^2:z:yw:w^2],
$$
by removing the last component $w^2$. This is a similar construction as the one for $\varphi$ displayed in \color{purple}\hyperlink{Table 4}{Table 4}\color{black}, which was obtained from the $3$-Veronese map $v_3$.

\begin{lemma}
The map $\psi$ is birational and it restricts to an isomorphism on the general anticanonical divisor of $\bbP$.
\end{lemma}

The proof that $\psi$ is birational, which we shall not detail here, consists in a resolution of the indeterminacy point of $\psi$, as was done in the proof of \color{purple}Lemma \ref{lem:example is birational}\color{black}. A similar argument applies to $\psi$ as the one which was given for $\varphi : \bbP\dashrightarrow \bbP(1,2,4,5)$ which ensured that $\varphi(S) \simeq S$ for a general $S\in |-K_\bbP|$. It revolves around the following commutative diagram. 

\begin{center}
\begin{tikzcd}
\bbP \arrow[d,hookrightarrow,"v_2"] \arrow[drrr,dashrightarrow,"\psi"] & & & \\
\bbP(1,3,5,9,15) \arrow[d,dashrightarrow,"|\Oo(15)|"] & & & \bbP(1,3,5,9) \arrow[d,dashrightarrow,"|\Oo(15)|"] \\
\mathbf W = \mathrm{cone}(\mathbf V) \arrow[rrr,dashrightarrow,"\mathrm{pr}"] & & & \mathbf V  
\end{tikzcd}
\end{center}
Here, $v_2$ is the $2$-Veronese map from $\bbP$ to $\bbP(1,3,5,9,15)$, $-K_\bbP = v_2^*\Oo_{\bbP(1,3,5,9,15)}(15)$, $\mathbf W$ is a cone over $\mathbf V$ with a point as vertex and $\mathrm{pr}$ is the projection map from the vertex point of $\mathbf W$ onto $\mathbf V$.

A consequence of this is that $S$ can be realized as a nongeneral anticanonical divisor of $\bbP(1,3,5,9)$, namely, one of equation
$$
v'f_{15}(u',v',s',t') + t'^2 = 0
$$
where $f_{15}$ is a homogeneous polynomial of degree $15$. One checks that the pullback to $\bbP$ of such a hypersurface is $S + (y^2)$, and the locus $y = 0$ is contracted by $\psi$.

In $\bbP(1,3,5,9)$, the curve $\Gamma$ is cut out on $S$ by a general hypersurface of degree $15$, as the diagram above shows. Hence $\Gamma$ is given in $\bbP(1,3,5,9)$ by the following equations
$$
v'f_{15}(u',v',s',t') + t'^2 = g_{15}(u',v',s',t') = 0
$$
where $g_{15}$ is a general homogeneous polynomial of degree $15$. Let's add two coordinates $r_0'$ and $r_1'$ of weight $15$ and examine the hypersurface $Y_2$ in $\mathbf X' = \bbP(1,3,5,9,15^2)$ given by the equation
$$
u'^3r_0' + v'r_1' = t'^2.
$$

\begin{lemma} \label{lem:2 3 10 15 nonmaximal extension 2}
The variety $Y_2$ has a model in $\bbP^{18}$ which is also a maximal extension of $\bbP$, although a nonmaximal extension of $\Gamma$.
\end{lemma}

\begin{proof}
It has dimension $4$ and is embedded in $\bbP^{18}$ by restriction of $|\Oo_{\mathbf X'}(15)|$; the very ampleness of $\mathcal O_{\mathbf X'}(15)|_{Y_2}$ being proven by a similar argument as \color{purple}Lemma \ref{lem:1 2 6 9 extension very ample}\color{black}. This model contains $\Gamma$ as a linear section; indeed, given two constants $\lambda_0$ and $\lambda_1$ : 
$$
Y_2\cap \left\{ \begin{array}{l}
r_0' = \lambda_0 g_{15}(u',v',s',t') \\ r_1' = \lambda_1 g_{15}(u',v',s',t') - f_{15}(u',v',s',t') \\ g_{15}(u',v',s',t') = 0
\end{array} \right\} = \Gamma.
$$
This is an extension of $\bbP$ as well. Indeed, we have $Y_2 \cap \left\{ r_0' = 0 \right\} = \left\{ v'r_1' = t'^2 \right\} = \bbP$ in $\bbP(1,3,5,9,15)$.
\end{proof}

\subsubsection{The maximal extension of $\bbP(2,3,10,15)$} \hypertarget{§4.5.2}{}

In the preceding subsections we constructed $Y_1$ and $Y_2$ two fourfold extensions of $\Gamma$. Now we construct a maximal extension $Y$ of $\bbP$, such that $Y$ contains both $Y_1$ and $Y_2$ as hyperplane sections in $\bbP^{19}$. This construction involves a weighted projective bundle over $\bbP^1$, i.e., a quotient of a vector bundle such that the fiber is a weighted projective space.

Let $\Lambda = \bbP^{17}$ be the linear subspace spanned by $\bbP$ in $\bbP^{19}$; for $i=1,2$ we have $Y_i = Y\cap H_i$ where $H_i$ is a hyperplane in $\bbP^{19}$ such that $\Lambda \subset H_i$. The fourfolds $Y_1$ and $Y_2$ generate a pencil of hyperplane sections of $Y$ which all contain $\bbP$.

The construction of $Y$ will require a realization of $Y_1$ and $Y_2$ as complete intersections in $\bbP(1^2,2,3,5^3)$. Note that the image of the $6$-Veronese map $v_6$ on $\bbP(2,3,10,15)$ lies in $\bbP(1^2,2,3,5^2)$, so one might think that it could be possible to recover $Y_1$ and $Y_2$ from $v_6$. However, all my attempts in trying so were unsuccessful.

On the one hand, $Y_1$ is given as the $12$-ic hypersurface in $\mathbf X$ of equation $u^2r_0 + vr_1 = s^3$. On the other hand, $Y_2$ is the $18$-ic hypersurface in $\mathbf X'$ given by the equation $u'^3r_0' + v'r_1' = t'^2$. Both $\mathbf X$ and $\mathbf X'$ can be embedded in $\bbP(1^2,2,3,5^3)$ by the following Veronese maps.
$$
\fct{(v_2)_{\mathbf X}}{\mathbf X=\bbP(1,2,4,5,10^2)}{\bbP(1^2,2,3,5^3)}{[u:v:s:t:r_0:r_1]}{[u^2:v:s:ut:t^2:r_1:r_0].}
$$
$$
\fct{(v_3)_{\mathbf X'}}{\mathbf X' = \bbP(1,3,5,9,15^2)}{\bbP(1^2,2,3,5^3)}{[u':v':s':t':r_0':r_1']}{[v':u'^3:u's':t':r_1':s'^3:r_0'].}
$$
We may choose $[U_0:U_1:V:W:X_0:X_1:X_2]$ as homogeneous coordinates on $\bbP(1^2,2,3,5^3)$, whose pullbacks by the Veronese maps are
$$
\begin{array}{c|ccccccc}
& U_0 & U_1 & V & W & X_0 & X_1 & X_2 \\
\hline
\text{pullback to } \mathbf X & u^2 & v & s & ut & t^2 & r_1 & r_0 \\
\text{pullback to } \mathbf X' & v' & u'^3 & u's' & t' & r_1' & s'^3 & r_0'
\end{array}
$$
Hence the above realizes $\mathbf X$ (respectively $\mathbf X'$) as the hypersurface of equation $U_0X_0 = W^2$ (respectively $U_1X_1 = V^3$). The descriptions we know for $Y_1$ and $Y_2$ in $\mathbf X$ and $\mathbf X'$ yield 
$$
Y_1 = \left\{ \begin{array}{l}
U_0X_0 = W^2 \\ U_1X_1 + U_0X_2 = V^3 
\end{array}
\right\}
$$
and
$$
Y_2 = \left\{ \begin{array}{l}
U_0X_0 + U_1X_2 = W^2 \\ U_1X_1 = V^3
\end{array}
\right\}.
$$
Besides, we know from \color{purple}Lemmas \ref{lem:2 3 10 15 nonmaximal extension 1} \color{black}and \color{purple}\ref{lem:2 3 10 15 nonmaximal extension 2} \color{black}that $\bbP$ is cut out on $Y_1$ and $Y_2$ by the same hyperplane in $\bbP^{18}$, namely: $Y_1 \cap \left\{X_2 = 0\right\} = Y_2 \cap \left\{ X_2 = 0 \right\} = \bbP$. In particular, $\bbP$ in $\left\{ X_2 = 0 \right\}$ is given by the equations 
\begin{equation} \label{eq:2 3 10 15 model}
U_0X_0 = W^2,\, U_1X_1 = V^3.
\end{equation}
We introduce now two coordinates $\lambda,\mu$ and consider $F = \mathrm{Proj}(R)$ with 
$$
R = \C[\lambda,\mu,U_0,U_1,V,W,X_0,X_1,X_2]
$$
endowed with the following grading in $\Z^2$:
$$
\begin{array}{c|ccccccccc}
 & \lambda & \mu & U_0 & U_1 & V & W & X_0 & X_1 & X_2 \\
\hline
\text{degree} & 1 & 1 & 0 & 0 & 0 & 0 & 0 & 0 & -1 \\
\text{in } \Z^2: & 0 & 0 & 1 & 1 & 2 & 3 & 5 & 5 & 5
\end{array}
$$
It is a bundle over $\bbP^1$ with fiber $\bbP(1^2,2,3,5^3)$, whose bundle map to $\bbP^1$ is $[\lambda : \mu]$ and the locus $X_2 = 0$ is the trivial subbundle $\bbP^1 \times \bbP(1^2,2,3,5^2)$.

There is a morphism $\phi : F \to \bbP(1^2,2,3,5^4)$ which is given in coordinates by the expression $[U_0:U_1:V:W:X_0:X_1:\lambda X_2:\mu X_2]$ and the projective model in $\bbP^{19}$ induced by the linear system $|\Oo_F(0,5)|$ decomposes as the composite map

\begin{center}
\begin{tikzcd}
F \arrow[rr,"\phi"] & & \bbP(1^2,2,3,5^4) \arrow[rr,dashrightarrow,"|\Oo(5)|"] & & \bbP^{19}.
\end{tikzcd}
\end{center}

Notice that $\phi$ contracts the trivial bundle $\bbP^1 \times \bbP(1^2,2,3,5^2)$ given by the equation $X_2 = 0$ onto $\bbP(1^2,2,3,5^2)$. Hence the image of $\left\{ X_2 = 0 \right\}$ by $|\Oo_F(0,5)|$ is the image of $\bbP(1^2,2,3,5^2)$ in $\bbP^{17}$.

Consider the complete intersection $Z$ in $F$ given by the two homogeneous equations 
\begin{align*}
U_0X_0 + \lambda U_1X_2 & = W^2, \\
U_1X_1 + \mu U_0X_2 & = V^3.
\end{align*}

\begin{lemma}\label{lem:2 3 10 15 max extension}
The image of $Z$ in $\bbP^{19}$ is not a cone and contains $Y_1$ and $Y_2$ as hyperplane sections. By \color{purple}\hyperlink{Table 2}{Table 2}\color{black}, it has dimension $1+\alpha(\Gamma,K_\Gamma)$, and thus it is a maximal extension of $\bbP$.
\end{lemma}

\begin{proof}
The restriction of $Z$ to $\left\{ X_2 = 0\right\}$ is the complete intersection in  $\bbP^1 \times \bbP(1^2,2,3,5^2)$ yielded by the equations $U_0X_0 = W^2$ and $U_1X_1 = V^3$. These are the defining equations for $\bbP$ in $\bbP(1^2,2,3,5^2)$ as mentioned in \color{purple}(\ref{eq:2 3 10 15 model})\color{black}, hence 
$$
Z\cap \left\{ X_2 = 0\right\} = \bbP^1 \times \bbP
$$
and it is contracted by $\phi$ to $\bbP$.

Let $Y$ be the image of $Z$ in $\bbP^{19}$. Let us show that it contains $Y_1$ and $Y_2$ as hyperplane sections. On the one hand, $\left\{ \lambda X_2 = 0\right\}$ is the pullback to $F$ of a hyperplane in $\bbP^{19}$, such that
$$
Z \cap \left\{ \lambda X_2 = 0 \right\} = Z|_{\lambda = 0} + Z|_{X_2 = 0}.
$$
In the above, $Z|_{X_2 = 0}$ is contracted onto $\bbP$, and $Z|_{\lambda = 0}$ has image $Y_1$. 

On the other hand,
$$
Z \cap \left\{ \mu X_2 = 0 \right\} = Z|_{\mu = 0} + Z|_{X_2 = 0}
$$
where once again, $Z|_{X_2 = 0}$ is contracted onto $\bbP$, and $Z|_{\mu = 0}$ has image $Y_2$.

It remains to be proven that $Y$ is not a cone. The pencil of fourfold extensions of $\bbP$ contained in $Y$ consists of all the $Y\cap H$, where $H\subset \bbP^{19}$ is a hyperplane such that $\bbP \subset H$. These fourfolds are each cut out on $Y$ by $\ell(\lambda,\mu)X_2 = 0$, with $\ell$ a linear form. Hence, they are complete intersections in $\bbP(1^2,2,3,5^3)$ of the form
\begin{align*}
U_0X_0 + \lambda U_1X_2 & = W^2, \\
U_1X_1 + \mu U_0X_2 & = V^3
\end{align*}
where $\lambda$ and $\mu$ are fixed constant coefficients (to be precise, solutions to $\ell(\lambda,\mu) = 0$). Let $Y_{(\lambda,\mu)}$ be the fourfold section of $Y$ given by the equations above, so that $Y_1 = Y_{(0,1)}$ and $Y_2 = Y_{(1,0)}$. We first notice that $Y_{(\lambda,\mu)} \simeq Y_{(\alpha\lambda,\beta\mu)}$ for all $\alpha,\beta \in \C^*$; indeed, the automorphism which consists in the change of variables $U_1 \mapsto \alpha U_1, U_0 \mapsto \beta U_0, X_1 \mapsto \frac{1}{\alpha}X_1$ and $X_0 \mapsto \frac{1}{\beta}X_0$ identifies $Y_{(\alpha\lambda,\beta\mu)}$ with $Y_{(\lambda,\mu)}$. Therefore, among the $Y_{(\lambda,\mu)}$ there are at most three isomorphism classes: $Y_{(1,0)},Y_{(0,1)}$ and $Y_{(1,1)}$. In particular, the class represented by $Y_{(1,1)}$ is dense in the pencil $\left\{ Y\cap H \: | \: \bbP \subset H\right\}$.

Assume now by contradiction that $Y$ is a cone. It contains $\bbP$ as a linear section of codimension $2$, and $\bbP$ is not a cone, so there are only two possible cases: either the vertex of $Y$ is a point, or it is a line. In the latter case, all the $Y_{(\lambda,\mu)}$'s are cones over $\bbP$ with each time a point as vertex; in the former case, there is a unique member $Y_{(\lambda,\mu)}$ which is a cone over $\bbP$. This unique member is either $Y_{(1,0)}$ or $Y_{(0,1)}$ since the class of $Y_{(1,1)}$ is dense in the pencil, so without loss of generality we may assume that $Y_{(1,0)}$ is a cone over $\bbP$ with a point as vertex (the rest of the proof is analogous if the cone is $Y_{(0,1)}$).

Let us recall the equations for $Y_{(1,0)}$ in $\bbP(1^2,2,3,5^3)$ with respect to the coordinates $[U_0:U_1:V:W:X_0:X_1:X_2]$.
\begin{align*}
U_0X_0 + U_1X_2 & = W^2, \\
U_1X_1 & = V^3.
\end{align*}
We recall as well the fact that $\bbP$ is the hyperplane section $Y_{(1,0)}\cap \left\{ X_2 = 0 \right\}$ in $\bbP^{18}$. There is a change of variable which fixes the hyperplane $\left\{ X_2 = 0 \right\}$ and moves the vertex point to $p_{X_2}$. This change of variables makes the affine chart $Y_{(1,0)}|_{X_2 = 1}$ an affine cone, i.e., it eliminates the variable $X_2$ from the equations above. 

Indeed, let $F = G = 0$ be the defining equations of a cone whose vertex point is $p_{X_2}$, such that $F$ and $G$ are two homogeneous sextics on $\bbP(1^2,2,3,5^3)$, not divisible by $X_2$, and set $f = F|_{X_2=1}, g = G|_{X_2=1}$. If one of them is not homogeneous, say $f$, then
$$
f = f_6 + \tilde{f}
$$
where $f_6 = f_6(U_0,U_1,V,W,X_0,X_1)$ is homogeneous of degree $6$ and $\tilde{f}$ has degree $5$ or less. By the fact that $\deg(\tilde{f}) < 6$ and $g$ and $f$ are sextics, we have $\tilde{f}\notin (f,g)$ and thus there exists a point $q = (U_0,U_1,V,W,X_0,X_1)$ in the affine cone $\left\{ f = g = 0\right\}$ such that $\tilde{f}(q) \neq 0$. From the condition $f(q) = 0$, we have the equality $f_6(q) = -\tilde{f}(q)$, and by the fact that $\left\{ f = g = 0\right\}$ is an affine cone, then for all $\lambda\in \C^*$ the point
$$
\lambda \cdot q = (\lambda U_0, \lambda U_1, \lambda^2 V,\lambda^3 W, \lambda^5 X_0, \lambda^5 X_1)
$$
also belongs to $\left\{ f = g = 0\right\}$. If $\lambda$ is general, we have
$$
f_6(\lambda\cdot q) = \lambda^6f_6(q) = -\lambda^6\tilde{f}(q) \neq -\tilde{f}(\lambda\cdot q)
$$
since the equality $\lambda^6\tilde{f}(q) = \tilde{f}(\lambda\cdot q)$ is a polynomial condition of degree $6$ on $\lambda$. This leads to the contradiction that $\lambda \cdot q \notin \left\{ f = g = 0\right\}$ and the conclusion that $f$ and $g$ are homogeneous.

As a consequence, there exists a transformation of the form 
$$
[U_0:U_1:V:W:X_0:X_1:X_2] \mapsto [A\mathbf U:V:W:M\mathbf X]
$$
where $A\in GL_2(\C)$ and $M\in GL_3(\C)$, which eliminates $X_2$ from the equations of $Y_{(1,0)}$. Let us denote
$$
A = \left( \begin{array}{cc}
a & b \\ c & d
\end{array} \right) \text{ and } M = \left( \begin{array}{ccc}
\alpha_0 & \beta_0 & \gamma_0 \\ \alpha_1 & \beta_1 & \gamma_1 \\ \alpha_2 & \beta_2 & \gamma_2
\end{array} \right).
$$
This change of variables applied to the equations of $Y_{(1,0)}$ yields
\begin{align*}
(aU_0+bU_1)(\alpha_0X_0+\beta_0X_1+\gamma_0X_2) + (cU_0+dU_1)(\alpha_2X_0+\beta_2X_1+\gamma_2X_2) & = W^2, \\
(cU_0+dU_1)(\alpha_1X_0+\beta_1X_1+\gamma_1X_2) & = V^3.
\end{align*}
By the fact that this does not involve the variable $X_2$, we have 
$$
a\gamma_0 + c\gamma_2 = c\gamma_1 = b\gamma_0 + d\gamma_2 = d\gamma_1 = 0
$$
in other words,
$$
\left( \begin{array}{cc}
a & c \\ b & d
\end{array} \right) \left( \begin{array}{c}
\gamma_0 \\ \gamma_2
\end{array} \right) = 0 \text{ and } \left( \begin{array}{cc}
a & c \\ b & d
\end{array} \right) \left( \begin{array}{c}
0 \\ \gamma_1
\end{array} \right) = 0.
$$
This contradicts either $\det(A) \neq 0$, or $\det(M)\neq 0$. The conclusion is that $Y_{(1,0)}$ is not a cone over $\bbP$ and therefore, $Y$ is not a cone
\end{proof}

\section{The primitive polarizations of the K3 surfaces}\hypertarget{section 5}{}

We recall that the index of the polarized K3 surface $(S,-K_\bbP|_S)$, which is denoted by $i_s$ in \color{purple}\hyperlink{Table 2}{Table 2}\color{black}, is the divisibility of $-K_\bbP|_S$ in the Picard group of $S$, i.e., the largest integer $r$ such that $-\frac{1}{r}K_\bbP|_S$ is a Cartier divisor on $S$. Here, $\Gamma$ is a general member of $|-K_\bbP|_S|$, and we introduce $C$ a general member of $|-\frac{1}{i_S}K_\bbP|_S|$, so that $\Gamma = i_SC$ in $\mathrm{Pic}(S)$.

In what follows, we go through all the cases $\#9$ to $\#14$ that are listed in \color{purple}\hyperlink{Table 2}{Table 2} \color{black}and give a geometric description of the curve $C$.

\subsection{$\bbP = \bbP(1,4,5,10)$}

According to \color{purple}\hyperlink{Table 2}{Table 2}\color{black}, the index $i_S$ is equal to $2$. The genus of $C$ is $6$. In \color{purple}\hyperlink{Table 5}{Table 5}\color{black}, $S$ is explicitly given as the quintic hypersurface $u_0f_4(\mathbf u,v) + u_1^5 = 0$ in $\bbP(1,1,1,2)$, with $\deg f_4 = 4$, and $\Gamma$ is cut out on $S$ by a quartic. Therefore $C = \frac{1}{2}\Gamma$ is cut out by a quadric, so its defining equations are
$$
u_0f_4(\mathbf u,v) + u_1^5 = g_2(\mathbf u,v) = 0
$$
with $g_2$ a general homogeneous quadric polynomial.

\begin{lemma}\label{lem:plane quintic inflection}
The curve $C$ is isomorphic to a plane quintic with a total inflection point, i.e., there is a line $\Delta$ which is tangent to $C$ in $\bbP^2$ and $C|_\Delta$ is a quintuple point.

Conversely, any such plane quintic can be realized as a member of $|-\frac{1}{2}K_\bbP|_{S'}|$ for a K3 surface $S'\in |-K_\bbP|$.
\end{lemma}

\begin{proof}
Up to scaling, we may choose $g_2(\mathbf u,v) = v - \alpha(u_0,u_1,u_2)$ where $\alpha$ is a conic. Hence $C$ is cut out by
$$
u_0f_4(u_0,u_1,u_2,v) + u_1^5 = 0,\, v = \alpha(u_0,u_1,u_2).
$$
Substituting $\alpha(u_0,u_1,u_2)$ for $v$ in the first equation naturally realizes $C$ as a quintic in $\bbP^2$ with coordinates $[u_0:u_1:u_2]$. Moreover, the restriction of $C$ to the line $u_0 = 0$ is a quintuple point. Let $\Delta = \left\{ u_0 = 0\right\}$,
$$
C|_\Delta = 5p
$$
where $p = \left\{ u_0 = u_1 = 0\right\}$ in $\bbP^2$. This is an inflection point of order $5$ of the curve $C$. The tangent cone of $C$ at this point is the reduced line $\Delta$ by generality of $f_4$ (the curve $C$ is indeed smooth, since it is a general hyperplane section of $S$ in $\bbP^6$ and $S$ has isolated singularities). 

Conversely, let $C'$ be such a plane quintic. Up to a choice of coordinates, $C'$ is given by an equation of the form $u_0g_4(u_0,u_1,u_2) + u_1^5 = 0$ with $\deg g_4 = 4$. It holds that $C'$ in its canonical model can be extended by a quintic surface in $\bbP(1,1,1,2)$, as the following points out. Following the construction that was done in \cite[Appendix A.2]{Lo}, there exists a quintic polynomial $f_5(\mathbf u,v)$ and a quadric $\alpha(\mathbf u) = \alpha(u_0,u_1,u_2)$ such that
$$
C' = \left\{ f_5(\mathbf u,v) = 0,\, v = \alpha(\mathbf u) \right\}
$$
in $\bbP(1,1,1,2)$. Hence the quintic surface $\left\{ f_5(\mathbf u,v) = 0\right\}$ in $\bbP(1,1,1,2)$ is an extension of $C'$.

Here $f_5$ and $\alpha$ are so that $f_5|_{v = \alpha(\mathbf u)} = u_0g_4(u_0,u_1,u_2) + u_1^5$. Thus $f_5 = u_1^5 + \lambda \beta(\mathbf u,v)(v - \alpha(\mathbf u)) \text{ (mod } u_0)$ in $\C[u_0,u_1,u_2,v]$ for some constant $\lambda$ and $\deg \beta = 3$. Picking $\lambda = 0$ doesn't change $C'$, and yields $f_5 = u_0f_4(\mathbf u,v) + u_1^5$ for some homogeneous quartic $f_4$ on $\bbP(1,1,1,2)$. 

Hence $C'$ is cut out on $S'$ by a quadric, where $S'$ is the quintic $u_0f_4(\mathbf u,v) + u_1^5 = 0$. Recall from \color{purple}\hyperlink{Table 4}{Table 4} \color{black}that $\varphi : \bbP \dashrightarrow \bbP(1,1,1,2)$ restricts to an isomorphism on the general member of $|-K_\bbP|$; here $S'$ is a member of $\mathcal L = \varphi(|-K_\bbP|)$ so it is isomorphic to a general anticanonical divisor of $\bbP$. Moreover, $C' = -\frac{1}{2}K_\bbP|_{S'}$ in $\mathrm{Pic}(S')$.
\end{proof}

\subsection{$\bbP = \bbP(1,2,6,9)$}

According to \color{purple}\hyperlink{Table 2}{Table 2}\color{black}, the index $i_S$ of the polarization $(S,-K_\bbP|_S)$ is equal to $3$. The curve $C$ has genus $4$. We know by \color{purple}\hyperlink{Table 5}{Table 5} \color{black}that $S$ is a degree $10$ hypersurface in $\bbP(1,1,3,5)$, of equation $u_0f_9(\mathbf u,v,s) + s^2 = 0$ and $\Gamma$ is the intersection of $S$ with a general $9$-ic. Hence $C$ is cut out on $S$ by a general cubic of $\bbP(1,1,3,5)$, i.e., its equations are
$$
u_0f_9(\mathbf u,v,s) + s^2 = 0,\, v = \alpha(u_0,u_1)
$$
where $\alpha$ is a homogeneous cubic polynomial on $\bbP^1$.

\begin{lemma}
The curve $C$ is isomorphic to a $10$-ic curve in $\bbP(1,1,5)$ i.e., a quadric section of the cone over a rational normal quintic curve. Equivalently, $C$ is a smooth hyperelliptic curve of genus $4$.

Conversely, any such curve can be realized as a member of $|-\frac{1}{3}K_\bbP|_{S'}|$ for a K3 surface $S'\in |-K_\bbP|$.
\end{lemma}

\begin{proof}
By the above equations, $C$ is naturally realized as a degree $10$ curve on $\bbP(1,1,5)$ with coordinates $[u_0:u_1:s]$ of the form $u_0h_9(\mathbf u,s) + s^2 = 0$. Hence the linear system $|\Oo_{\bbP(1,1,5)}(1)|$, whose base locus $\left\{ u_0 = u_1 = 0\right\}$ does not meet $C$, restricts to a $g_2^1$ on $C$.

Conversely, let $C'$ be a curve in $\bbP(1,1,5)$ of degree $10$. In a suitable choice of coordinates, the point $[u_0:u_1] = [0:1]$ belongs to the branch locus of the double cover $C\to \bbP^1$. Hence the line $\Delta = \left\{ u_0 = 0 \right\}$ in $\bbP(1,1,5)$ is tangent to $C'$. As a result, the restriction $C'|_{u_0 = 1}$ is a double point, which yields the following equation for $C'$:
$$
u_0h'_9(\mathbf u,s) + s^2 = 0,
$$
with $\deg h'_9 = 9$. Introducing $v$ a coordinate of weight $3$ and $f'_9(\mathbf u,v,s)$ a degree $9$ homogeneous polynomial on $\bbP(1,1,3,5)$ such that $f'_9(\mathbf u,v,s)|_{v = \alpha(u_0,u_1)} = h'_9(\mathbf u,v)$, we realize $C'$ as a complete intersection in $\bbP(1,1,3,5)$ of equations
$$
u_0f'_9(\mathbf u,v,s) + s^2 = 0,\, v = \alpha(u_0,u_1).
$$
This makes $C'$ a curve in the surface $S' = \left\{ u_0f'_9(\mathbf u,v,s) + s^2 = 0\right\}$, which is a member of $\mathcal L = \varphi(|-K_\bbP|)$, meaning that $S'$ is isomorphic to a general anticanonical divisor of $\bbP$. Moreover, the moving part of $\mathcal L|_{S'}$ is the restriction to $S'$ of the $9$-ics, and therefore $3C$ is a member of the moving part of $\mathcal L|_S$. This makes $3C$ the class of the hyperplane sections of $S'$ in $\bbP^{28}$, in other words $3C = -K_\bbP|_{S'}$ in $\mathrm{Pic}(S')$. This yields $C = -\frac{1}{3}K_\bbP|_{S'}$.
\end{proof}

\subsection{$\bbP = \bbP(1,2,3,6)$}

This is the only example of our list for which $-\frac{1}{i_S}K_\bbP$ is Cartier. The index $i_S$ is equal to $2$ and $-\frac{1}{2}K_\bbP$ is the class of sextic surfaces. Hence, $S$ is a general surface of degree $12$ and $C$ is cut out on $S$ by a sextic. The projective model associated to $-\frac{1}{2}K_\bbP$ in which $C$ is a hyperplane section of $S$ is a realization of $\bbP$ as a variety of degree $(-\frac{1}{2}K_\bbP)^3 = 6$ in $\bbP^7$. It factors as the composite map 
\begin{center}
\begin{tikzcd}
\bbP \arrow[rr,hookrightarrow,"v_2"] & & \bbP(1,1,2,3,3) \arrow[rr,dashrightarrow,"|\Oo(3)|"] & & \bbP^7
\end{tikzcd}
\end{center}
where $v_2$ is the $2$-Veronese embedding mentioned in \color{purple}\hyperlink{Table 3}{Table 3}\color{black}.

Let $[u_0:u_1:v:s_0:s_1]$ be coordinates on $\bbP(1,1,2,3,3)$, then $v_2$ realizes $\bbP$ as the hypersurface $u_0s_0 = v^2$, $S$ as the intersection of $\bbP$ with a general sextic, and $C$ as the intersection of $S$ with a general cubic.

Consider now $\bbP' := \bbP(1,1,1,3)$ with coordinates $[a_0:a_1:a_2:b]$ and the rational map $\psi : \bbP' \dashrightarrow \bbP(1,1,2,3,3)$ given by the expression 
$$
[u_0:u_1:v:s_0:s_1] = [a_0:a_1:a_0a_2:a_0a_2^2:b].
$$ 
Its image satisfies the same equation as $\bbP$, hence it is equal to $\bbP$. There is a birational map $\varphi$ which makes the following diagram commute.
\begin{center}
\begin{tikzcd}
\bbP \arrow[rr,dashrightarrow,"\varphi"] \arrow[rd,hookrightarrow,"v_2",swap] & & \bbP' \arrow[ld,dashrightarrow,"\psi"] \\
 & \bbP(1,1,2,3,3) & 
\end{tikzcd}
\end{center}
One checks from the expression of $v_2$ in \color{purple}\hyperlink{Table 3}{Table 3} \color{black}and that of $\psi$ that $\varphi$ has the following expression with regard to the weighted coordinates.
$$
\varphi : [x:y:z:w] \mapsto [a_0:a_1:a_2:b] = [x^3:xy:z:x^3w].
$$
The rational map $\varphi$ admits a rational inverse $\varphi^{-1}$, which is: 
$$
\varphi^{-1} : [a_0:a_1:a_2:b] \in \bbP' \mapsto [x:y:z:w] = [a_0:a_0a_1:a_0^2a_2:a_0^3b]\in \bbP.
$$
Indeed, a computation gives
$$
\varphi \circ \varphi^{-1}: [a_0:a_1:a_2:b] \mapsto [a_0^3:a_0^2a_1:a_0^2a_2:a_0^6b] = [a_0:a_1:a_2:b].
$$
Let us consider $q = [a_0:a_1:a_2:b]$ a fixed point in $\bbP'$ with a chosen representative $(a_0,a_1,a_2,b)\in \C^4$, and $\sqrt{a_0}$ a square root of $a_0$. Then the image of $q$ by $\varphi^{-1}$ is
\begin{equation}
\label{eq:phiinverse}
\varphi^{-1}(q) = [\sqrt{a_0}:a_1:\sqrt{a_0}a_2:b] \in \bbP,
\end{equation}
and a computation yields 
$$
v_2 \circ \varphi^{-1}(q) = \psi(q),
$$
as required in the commutative diagram.

The map $\varphi$ contracts the vanishing locus of $x$ to a point $p$. The divisor $D = \left\{ x=0\right\}$ has degree $2$ on $C$; indeed, $C$ is cut out on $\bbP$ by two general equations of respective degree $12$ and $6$ with regard to the grading of $\bbP$, and $D\in |\Oo_\bbP(1)|$, therefore:
$$
\deg D|_C = D\cdot C = [\Oo_\bbP(1)]\cdot [\Oo_\bbP(12)]\cdot [\Oo_\bbP(6)] = 2.
$$
This ensures that the restriction $\varphi|_C$ maps $2$ distinct points to $p$. The indeterminacy locus $x=z=0$ does not meet $C$ by the generality assumption, so the map $\varphi$ induces by restriction to $C$ a morphism $C\to \bbP(1,1,1,3)$ which has degree $1$ and makes $C$ the normalization of its image. Let $C_0$ be the image of $C$ by $\varphi$.

\begin{lemma}
\label{lem:smoothbranches}
The curve $C_0$ has a unique singularity at which two smooth local branches meet. The morphism $C\to C_0$ is an isomorphism over the complement of this singular point.
\end{lemma}

\begin{proof}
From the expression  of $\varphi$ in coordinates, we identify that its exceptional locus is $D = \left\{ x = 0 \right\} \subset \bbP$ and $\varphi(D)$ is the point $p = [0:0:1:0] \in \bbP' = \bbP(1^3,3)$. Moreover, $C$ is smooth and the transverse intersection $D\cap C$ consists of two points. From the expression \color{purple}(\ref{eq:phiinverse})\color{black}, we identify that the only indeterminacy point of the rational inverse $\varphi^{-1}$ of $\varphi$ is $p$. Since $D$ is the preimage of $p$ by $\varphi$ and $\varphi \circ \varphi^{-1} = \mathrm{id}$ on the regular locus of $\varphi^{-1}$, we have $\mathrm{Exc}(\varphi) = D$, and moreover $\varphi$ is an isomorphism on its regular locus minus $D$, i.e., $\varphi$ induces by restriction an isomorphism from $\bbP - (D\cup \left\{ x=z=0 \right\})$ onto its image.

As a result, the morphism $\varphi|_C : C \to C_0$ is an isomorphism outside of $D\cap C$, and $C_0$ has only one singular point which is the intersection of two local branches. To prove that these two local branches are smooth, it suffices to show that the kernel of the differential $d\varphi$ coincides with the tangent of $D$ at any of the two points $D\cap C$.

Consider the open subset $U = \left\{ z \neq 0 \right\} \subset \bbP(1,2,3,6)$. By generality of $C$, this open subset contains the two points $D\cap C$. On the one hand $U$ is isomorphic to the quotient of $\C^3_{(\mathtt x,\mathtt y,\mathtt w)}$ under the following action of the group of cubic roots of unity $\mu_3$:
$$
\zeta \cdot (\mathtt x,\mathtt y,\mathtt w) = (\zeta \mathtt x,\zeta^2 \mathtt y,\mathtt w).
$$
On the other hand, the indeterminacy locus $\left\{ x=z=0 \right\}$ of $\varphi$ is disjoint from $U$, and the image of $U$ via $\varphi$ is the open subset $\left\{ a_2 \neq 0 \right\}$ of $\bbP(1^3,3)_{[a_0:a_1:a_2:b]}$, which is isomorphic to $\C^3_{(\mathtt a_0,\mathtt a_1,\mathtt b)}$.

Consider the map $\overline \varphi : \C^3_{(\mathtt x,\mathtt y,\mathtt w)} \to \C^3_{(\mathtt a_0,\mathtt a_1,\mathtt b)}$ given by the expression
$$
\overline \varphi : (\mathtt x,\mathtt y,\mathtt w) \mapsto (\mathtt x^3,\mathtt xy, \mathtt x^3\mathtt w).
$$
This map is invariant under the action of $\mu_3$ and it fits in the following commutative diagram, where the left vertical arrow is the quotient map onto $U \simeq \sfrac{\C^3}{\mu_3}$, which is generically a local diffeomorphism:
\begin{center}
\begin{tikzcd}
\C^3 \arrow[d,"3:1",swap] \arrow[rr,"\overline \varphi"] & & \C^3 \arrow[d,"\simeq"] \\
U \arrow[rr,"\varphi",swap] & & \varphi(U)
\end{tikzcd}
\end{center}
Let us denote by $C', D'\subset \C^3$ the respective preimages of $C\cap U$ and $D\cap U$ by the quotient map $\C^3 \to U$. Let $p_1,p_2$ be the two intersection points of $D$ and $C$; by generality, each point $p_i$ has three distinct preimages in $D'\cap C'$, and above the point $p_i$ the quotient map $\C^3 \to U$ is a local diffeomorphism. Hence at the general point $\mathtt p \in D' \cap C'$, it is enough to show that
$$
\mathrm{ker}(d_\mathtt{p} \overline \varphi) = T_\mathtt{p} D'.
$$
The equality above follows from a straightforward computation, as $D'$ is the vanishing locus of $\mathtt x$ and $\overline \varphi$ is the map $(\mathtt x,\mathtt y,\mathtt w) \mapsto (\mathtt x^3,\mathtt x \mathtt y,\mathtt x^3 \mathtt w)$.
\end{proof}

\begin{lemma}\label{lem:plane sextic oscnode}
The curve $C_0$ is isomorphic to a plane sextic with an oscnode at a point $p$, i.e., $C_0$ has two smooth local branches meeting with contact order $3$ at its singular point $p$. Moreover, there is a line $\Delta \subset \bbP^2$ through $p$ such that $C_0|_\Delta = 6p$. 

Besides, $C$ is the normalization of $C_0$. 
\end{lemma}

\begin{proof}
We know by \color{purple}(\ref{eq:phiinverse}) \color{black}that the expression of $\varphi^{-1}$ is
$$
\varphi^{-1}([a_0:a_1:a_2:b]) = [\sqrt{a_0}:a_1:\sqrt{a_0}a_2:b] \in \bbP
$$
Let $\Sigma$ be the direct image of $S$ under $\varphi$; it is the proper transform of $S$ by $\varphi^{-1}$. From the above, we identify the exceptional locus of $\varphi^{-1}$ as $\left\{ a_0 = 0\right\}$ and the pullback to $\bbP'$ of the general $12$-ic surface $S$ is
$$
(\varphi^{-1})^*S = \left\{ a_0^6(a_0f_5(\mathbf a,b) + \lambda a_1^6 + \mu a_1^3b + \gamma b^2) = 0 \right\}
$$
where $f_5(\mathbf a,b)$ is a quintic on $\bbP'$ and $\lambda,\mu,\gamma$ are constant. Hence the proper transform $\Sigma$ of $S$ is a nongeneral sextic of $\bbP' = \bbP(1,1,1,3)$ of equation
$$
a_0f_5(\mathbf a,b) + \lambda a_1^6 + \mu a_1^3b + \gamma b^2 = 0.
$$
On the one hand, we have $h^0(\bbP,\Oo_\bbP(-K_\bbP)) = 27$, while $h^0(\bbP',\Oo_{\bbP'}(5)) + 3 = 30$, so even the quintic $f_5$ can't be general, and it must belong to a subspace $V\subset H^0(\bbP',\Oo_{\bbP'}(5))$ with $\dim V = \dim H^0(\bbP',\Oo_{\bbP'}(5))-3 = 24$. Namely, it follows from the expression of $\varphi^{-1}$ that $f_5$ does not involve the monomials $a_2^5,a_2^4a_1$ and $a_2^3a_1^2$.

As $C$ is a hyperplane section of $S$ in $\bbP^7$, it is cut out on $S$ by a general cubic $\alpha(u_0,u_1,v,s_0,s_1) = 0$ in $\bbP(1,1,2,3,3)$, with
$$
[u_0:u_1:v:s_0:s_1] = [a_0:a_1:a_0a_2:a_0a_2^2:b],
$$
and thus its image $C_0$ in $\bbP'$ is cut out on $\Sigma$ by a nongeneral cubic of the form 
$$
b = \tau a_1^3 + a_0q(\mathbf a)
$$
with $\tau$ a constant and $q$ a general quadric.
$$
C_0 = \Sigma \cap \left\{ b = \tau a_1^3 + a_0q(\mathbf a) \right\} = \left\{ \begin{array}{l}
a_0f_5(\mathbf a,b) + \lambda a_1^6 + \mu a_1^3b + \gamma b^2 = 0 \\
b = \tau a_1^3 + a_0q(\mathbf a)
\end{array} \right\}.
$$
This makes $C_0$ a sextic curve in $\bbP^2$ with coordinates $[a_0:a_1:a_2]$ such that $C_0|_{a_0=0} = (a_1^6)$. Hence the restriction of $C_0$ to the line $\Delta = \left\{ a_0=0 \right\}$ is a sextic point. This point is the intersection point of $C_0$ with $\Delta$, i.e., the point $p = \left\{ a_0 = a_1 = 0 \right\}$.

By \color{purple}Lemma \ref{lem:smoothbranches}\color{black}, the curve $C_0$ has a unique singular point which is the intersection of two local branches, and the morphism from $C$ to $C_0$ is finite and birational. Since it is smooth, we deduce that $C$ is the normalization of $C_0$. Moreover, the curve $C_0$ has two smooth local branches at $p$, say $B_1$ and $B_2$, such that 
$$
6p = C_0|_\Delta = B_1|_\Delta + B_2|_\Delta.
$$
This implies $B_i|_\Delta = \beta_ip$ with $\beta_i \in \mathbf N$ for $i = 1,2$, and $\beta_1 + \beta_2 = 6$. There are three cases to distinguish.
\begin{enumerate}
\item[$(i)$] If $(\beta_1,\beta_2) = (1,5)$, then $C_0$ is a sextic plane curve with a node at $p$, and $g(C) = \frac{(6-1)(6-2)}{2} - 1 = 9$. But $g(C) = 7$ by \color{purple}\hyperlink{Table 2}{Table 2}\color{black}.
\item[$(ii)$] If $(\beta_1,\beta_2) = (2,4)$, then $C_0$ has a tacnode at $p$, and $g(C) = \frac{(6-1)(6-2)}{2} - 2 = 8$, which is also a contradiction.
\item[$(iii)$] If $(\beta_1,\beta_2) = (3,3)$, then $C_0$ has an oscnode at $p$, and we indeed have $g(C) = \frac{(6-1)(6-2)}{2} - 3 = 7$.
\end{enumerate}
The conclusion follows that $p$ is an oscnodal point of $C_0$, as required.
\end{proof}

\subsection{$\bbP = \bbP(1,3,8,12)$}

As stated in \color{purple}\hyperlink{Table 2}{Table 2}\color{black}, the curve $C$ has genus $7$. It follows from \color{purple}\hyperlink{Table 5}{Table 5} \color{black}that $S$ is isomorphic to the $9$-ic hypersurface $u_0f_8(\mathbf u,v,s) + v^3 = 0$ in $\bbP(1,1,3,4)$ with coordinates $[u_0:u_1:v:s]$ and $\Gamma$ is cut out on $S$ by a degree $8$ hypersurface of $\bbP(1,1,3,4)$. The index $i_S$ is equal to 2, therefore $C = \frac{1}{2}\Gamma$ in $\mathrm{Pic}(S)$ and $C$ is the intersection of $S$ with a general quartic. Such a quartic has equation $s = \alpha(\mathbf u,v)$, where $\deg \alpha = 4$. Hence $C$ is cut out by the equations
$$
u_0f_8(\mathbf u,v,s) + v^3 = 0,\, s = \alpha(\mathbf u,v).
$$

\begin{lemma} \label{lem:trigonal genus 7}
The curve $C$ is isomorphic to a degree $9$ curve in $\bbP(1,1,3)$, i.e., a cubic section of the cone over a rational normal cubic curve, with an inflection point of order $3$ along a line of the ruling, i.e., there is a line $\Delta \in |\Oo_{\bbP(1,1,3)}(1)|$ which is tangent to $C$ at a point $p$, and $C|_\Delta = 3p$. In particular, $C$ is a trigonal curve of genus $7$ with a total ramification point.

Conversely, any such curve in $\bbP(1,1,3)$ is isomorphic to a member of $|-\frac{1}{2}K_\bbP|_{S'}|$ for a K3 surface $S'\in |-K_\bbP|$.
\end{lemma}

\begin{proof}
By the above equations, $C$ is naturally realized as the curve of degree $9$ in $\bbP(1,1,3)$ given by the following
$$
u_0h_8(\mathbf u,v) + v^3 = 0
$$
where $\deg h_8 = 8$. Let $\Delta$ be the line $\left\{ u_0 = 0\right\}$, then $C|_\Delta = 3p$ where $p$ is the point $\left\{ u_0 = v = 0\right\}$. The tangent cone of $C$ at $p$ is the reduced line $\Delta$, hence $C$ is smooth and has an inflection point of order $3$ at $p$.

Conversely, let $C' \subset \bbP(1,1,3)$ be such a curve. Then for a fitting choice of coordinates, it is cut out by an equation of the form 
$$
u_0h'_8(\mathbf u,v) + v^3 = 0
$$
with $\deg h'_8 = 8$. We introduce a coordinate $s$ of weight $4$ and a homogeneous degree $8$ polynomial $f'_8(\mathbf u,v,s)$ on $\bbP(1,1,3,4)$ such that $f'_8(\mathbf u,v,s)|_{s = \alpha(\mathbf u,v)} = h'_8(\mathbf u,v)$. In this setting, $C'$ is the complete intersection in $\bbP(1,1,3,4)$ given by
$$
u_0f'_8(\mathbf u,v,s) + v^3 = 0,\, s = \alpha(\mathbf u,v)
$$
meaning it is cut out by a general quartic on the surface $S' = \left\{ u_0f'_8(\mathbf u,v,s) + v^3 = 0\right\}$. Recall from \color{purple}\hyperlink{Table 4}{Table 4} \color{black}and \color{purple}\hyperlink{Table 5}{Table 5} \color{black}that the birational map $\varphi$ from $\bbP$ to $\bbP(1,1,3,4)$ restricts to an isomorphism on the general anticanonical divisors of $\bbP$. Here $S'$ is a member of $\mathcal L = \varphi(|-K_\bbP|)$, hence it is isomorphic to a general member of $|-K_\bbP|$, and furthermore the moving part of $\mathcal L|_{S'}$ is the restriction to $S'$ of the $8$-ics of $\bbP(1,1,3,4)$. As a result, $2C$ is a member of the moving part of $\mathcal L|_{S'}$, i.e., it is a hyperplane section of $S'$ in $\bbP^{25}$. Hence $2C = -K_\bbP|_{S'}$ in $\mathrm{Pic}(S')$, and the conclusion follows that $C = -\frac{1}{2}K_\bbP|_{S'}$.
\end{proof}

\subsection{$\bbP = \bbP(1,6,14,21)$}

We know from \color{purple}\hyperlink{Table 5}{Table 5} \color{black}that $S$ is a heptic hypersurface in $\bbP(1,1,2,3)$ with coordinates $[u_0:u_1:v:s]$ of equation $u_0f_6(\mathbf u,v,s) + u_1^7 = 0$. In this case, the index $i_S$ is equal to $1$, hence $C$ and $\Gamma$ are two curves of  genus $22$ which represent the same Cartier divisor on $S$, which is cut out by a general sextic of $\bbP(1,1,2,3)$. Such a sextic is smooth by generality, since $\bbP(1,1,2,3)$ has only two isolated singularities and the linear system of its sextics doesn't have base points. Moreover, the general sextic is a double cover of $\bbP(1,1,2)$ ramified over a general curve of degree $6$. It is indeed given by an equation of the form $s^2 = h_6(\mathbf u,v)$ with $\deg(h_6) = 6$, and the ramification locus in $\bbP(1,1,2)$ is the curve $h_6(\mathbf u,v) = 0$. This sextic of $\bbP(1,1,2,3)$ is thus a Del Pezzo surface of degree $1$ and we shall denote it by $DP_1$. In particular, it can be obtained from $\bbP^2$ by blowing up $8$ general points.

\begin{lemma}
The curve $C$ is the blowup of a degree $21$ plane curve $C_0$ at $8$ heptuple points $p_1,...,p_8$. Moreover, if $p$ is the ninth base points of the pencil $\mathcal P$ of plane cubics through the points $p_i$, then there exists $\gamma$ a member of $\mathcal P$ such that
$$
C_0|_\gamma = 7p + 7p_1 + \cdots + 7p_8.
$$
Conversely, the proper transform in $DP_1$ by the blowup map $DP_1 \to \bbP^2$ of any such plane curve $C_0$ of degree $21$ is isomorphic to a member of $|-K_\bbP|_{S'}|$ for a K3 surface $S'\in |-K_\bbP|$.
\end{lemma}

\begin{proof}
Let $\varepsilon : DP_1 \to \bbP^2$ be the blowup map, $H= \varepsilon^*\Oo_{\bbP^2}(1)$ the pullback of the lines and $E_i$ the exceptional curve over $p_i$, $i\in \{ 1,...,8 \}$. On the one hand, by the discrepancy of $\varepsilon$, we have $-K_{DP_1} = 3H - \sum_{i=1}^8 E_i$. On the other hand, the adjunction formula yields $-K_{DP_1} = \Oo_{\bbP(1,1,2,3)}(1)|_{DP_1}$. Since $C = DP_1 \cap S$, where $S$ is a heptic in $\bbP(1,1,2,3)$, it holds that 
$$
C = -7K_{DP_1} = 21H - \sum_{i=1}^8 7E_i
$$ 
and thus it is the proper transform of a degree $21$ curve $C_0$ in $\bbP^2$ which passes through the points $p_i$, each with multiplicity $7$.

The curve $C$ is given by the two following equations.
$$
u_0f_6(\mathbf u,v,s) + u_1^7 = g_6(\mathbf u,v,s) = 0
$$
for $g_6$ a general homogeneous sextic polynomial, so that $g_6 = 0$ is the defining equation of $DP_1$. The base point of $-K_{DP_1} = \mathcal O_{\bbP(1,1,2,3)}(1)|_{DP_1}$ is the intersection point of $DP_1$ with the locus $\left\{ u_0 = u_1 = 0\right\}$, which we denote by $p$. Let $B$ be the curve $DP_1 \cap \left\{ u_0 = 0 \right\}$. It is an anticanonical curve of $DP_1$ and by the equations above we have 
$$
C|_B = (DP_1\cap \left\{ u_0f_6(\mathbf u,v,s)+u_1^7 = 0 \right\})|_{u_0=0} = DP_1|_{u_0 = 0} \cap (u_1^7)|_{u_0=0} = 7p.
$$
If $\gamma = \varepsilon(B)$, which is a plane cubic through $p,p_1,...,p_8$, then the above implies that $C_0|_\gamma = 7p + 7p_1 + \cdots + 7p_8$.

Conversely, if $C_0'$ is a plane $21$-ic curve, then its blowup $C'\subset DP_1$ at the points $p_i$ is in the Cartier class $-7K_{DP_1}$, and by the surjectivity of
$$
H^0 (\bbP(1,1,2,3),\Oo_{\bbP(1,1,2,3)}(7)) \twoheadrightarrow H^0(DP_1,\Oo_{DP_1}(-7K_{DP_1}))
$$
which follows from the restriction short exact sequence 
$$
0 \to \Oo_{\bbP(1,1,2,3)} \to \Oo_{\bbP(1,1,2,3)}(7) \to \Oo_{DP_1}(-7K_{DP_1}) \to 0
$$
and the vanishing $h^1(\Oo_{\bbP(1,1,2,3)}) = 0$, we have $C' = S' \cap DP_1$ where $S'$ is a heptic surface in $\bbP(1,1,2,3)$. It follows that $C'$ in $\bbP(1,1,2,3)$ has equations
$$
f'_7(\mathbf u,v,s) = g_6(\mathbf u,v,s) = 0.
$$
Besides, there exists $B'$ an anticanonical curve of $DP_1$ such that $C'|_{B'} = 7p$. We may choose the coordinates $[u_0:u_1:v:s]$ on $\bbP(1,1,2,3)$ such that $B' = DP_1 \cap \left\{ u_0 = 0 \right\}$. This yields
$$
f'_7(\mathbf u,v,s)|_{u_0 = g_6(\mathbf u,v,s) = 0} = u_1^7.
$$
In other words, $f'_7 = u_1^7 + \lambda \alpha(u_0,u_1)g_6(\mathbf u,v,s) (\text{mod } u_0)$ for some constant $\lambda$ and $\deg \alpha = 1$. We may choose $\lambda = 0$, which does not change $C'$ and realizes it as a complete intersection in $\bbP(1,1,2,3)$ of the form
$$
u_0f'_6(\mathbf u,v,s) + u_1^7 = g_6(\mathbf u,v,s) = 0
$$
with $\deg f'_6 = 6$. Thus $C'$ lies on the surface $S' = \left\{ u_0f'_6(\mathbf u,v,s) + u_1^7 = 0 \right\}$. It is a member of $\mathcal L = \varphi(|-K_\bbP|)$, for $\varphi$ the birational map displayed in \color{purple}\hyperlink{Table 4}{Table 4}\color{black}. Therefore $S'$ is isomorphic to a general member of $|-K_\bbP|$ and $C' = -K_\bbP|_{S'}$.
\end{proof}

\subsection{$\bbP = \bbP(2,3,10,15)$}

The index $i_S$ is equal to $1$, meaning that both $C$ and $\Gamma$ represent the same Cartier divisor on $S$. The curve $C$ is then the intersection of $\bbP$ in $\bbP^{17}$ with two general hyperplanes. Recall from \color{purple}(\ref{eq:2 3 10 15 model}) \color{black}that $\bbP$ is realized as a complete intersection in $\bbP(1^2,2,3,5^2)$ with coordinates $[U_0:U_1:V:W:X_0:X_1]$ of equations 
$$
U_0X_0 = W^2,\, U_1X_1 = V^3
$$
and that its hyperplane sections in $\bbP^{17}$ are its sections by the quintics of $\bbP(1^2,2,3,5^2)$. To lighten the notation, let us use lower case letters instead of upper case ones to designate the coordinates, as there is no risk of confusion here.

The equations that cut out the curve $C$ in $\bbP$ are thus general quintics of $\bbP(1^2,2,3,5^2)$, and by the generality assumption we may choose them to be $x_0 = f_5(\mathbf u,v,w)$ and $x_1 = h_5(\mathbf u,v,w)$ where $f_5$ and $h_5$ are homogeneous of degree $5$.

This makes $C$ the curve in $\bbP(1,1,2,3)$ given by the two following sextic equations
$$
v^3 = u_1h_5(\mathbf u,v,w),\, w^2 = u_0f_5(\mathbf u,v,w).
$$
In other words, $C$ is the intersection of the two sextic surfaces $\Sigma = \left\{ v^3 = u_1h_5(\mathbf u,v,w) \right\}$ and $\Theta = \left\{ w^2 = u_0f_5(\mathbf u,v,w)\right\}$.

\begin{lemma}
\label{lem:Sigmabirational}
Consider the map $\varepsilon : \bbP(1,1,2,3) \dashrightarrow \bbP(1,1,2)$ given by 
$$
[u_0:u_1:v:w] \mapsto [u_0:u_1:v].
$$
The restriction of $\varepsilon$ to $\Sigma$ yields a birational map $\Sigma \dashrightarrow \bbP(1,1,2)$ which contracts the rational curve $\mathfrak f = \left\{ u_1 = v = 0 \right\}$ to the point $p_{u_0} = [1:0:0] \in \bbP(1,1,2)$. In addition, the restriction of $\varepsilon$ to $\Sigma - \mathfrak f$ is an isomorphism onto its image.
\end{lemma}

\begin{proof}
First of all, we identify the unique indeterminacy point of $\varepsilon$ as $p_w = [0:0:0:0:1]$.

The fiber of $\varepsilon$ over a smooth point of $\bbP(1,1,2)$ (say, over $[u_0:u_1:v] \neq[0:0:1]$) is a $\bbP(1,3)$ in $\bbP(1,1,2,3)$ parametrized by
$$
[\mathtt u:\mathtt w] \in \bbP(1,3) \to [\mathtt u u_0:\mathtt u u_1:\mathtt u^2 v:\mathtt w] \in \varepsilon^{-1}([u_0:u_1:v]).
$$
In addition, in this parametrization, the indeterminacy point of $\varepsilon$ along this fibre is $[0:1]$.

It follows from the defining equation of $\Sigma$ in $\bbP(1,1,2,3)$ that its restriction to such a fiber is the locus in $\bbP(1,3)$ given by an equation of the form $\mathtt u^3(\mu \mathtt u^3+\mathtt w) = 0$, with $\mu$ a constant. Via the isomorphism 
$$
[\mathtt u:\mathtt w] \in \bbP(1,3) \mapsto [\mathtt u^3:\mathtt w] \in \bbP^1,
$$
we see that $\Sigma$ has degree $2$ on this fibre and it cuts out (transversally) two disctinct points, one of which is the indeterminacy point $p_w$ of $\varepsilon$ (which is parametrized by $[\mathtt u:w] = [0:1]$). The only exception is the particular fiber $\mathfrak f := \left\{u_1 = v = 0\right\}$, which is a $\bbP(1,3)$ with coordinates $[u_0:w]$, contained in $\Sigma$ and contracted to the smooth point $p_{u_0} \in \bbP(1,1,2)$.

Meanwhile, the fiber over the singular point $p_v \in \bbP(1,1,2)$, consists of all the points $[u_0:u_1:v:w]$ for which $(u_0,u_1) = (0,0)$ and thus is a $\bbP(2,3)$ with coordinates $[v:w]$, and the restriction of $\Sigma$ to this particular fiber is given by the equation $v^3 = 0$. Hence $\Sigma$ meets this particular fiber only at the indeterminacy point of $\varepsilon$.

Therefore, $\varepsilon|_\Sigma$ is a birational map from $\Sigma$ to $\bbP(1,1,2)$ with indeterminacy point $p_w = \left\{ u_0 = u_1 = v = 0\right\}$ and its restriction to $\Sigma - \left\{ p_w \right\}$ is a birational morphism 
$$
\Sigma - \left\{ p_w\right\} \to \bbP(1,1,2) - \left\{ p_v\right\}
$$ 
which contracts the curve $\mathfrak f$ to the point $p_{u_0}$ and is an isomorphism on $\Sigma - \mathfrak f$.
\end{proof}

\begin{lemma}
\label{lem:smoothbranches2}
At the general point $q$ of $\mathfrak f$, we have
$$
\mathrm{ker}(d_q\varepsilon) = T_q \mathfrak f.
$$
\end{lemma}

The proof relies on a technical calculation which is similar to the one which was done in \color{purple}Lemma \ref{lem:smoothbranches}\color{black}.

\begin{lemma}
The curve $C$ is the normalization of a $1$-nodal curve $C_0$ in $\bbP(1,1,2)$ of degree $12$, i.e., a sextic section of the cone over a conic. Let $p$ be the node of $C_0$ such that $C = Bl_pC_0$, then it is a smooth point of $\bbP(1,1,2)$ and the line $\Delta \in |\Oo_{\bbP(1,1,2)}(1)|$ through $p$ is such that $C_0|_\Delta = 6p$. Furthermore, there is another line $\Delta' \in |\Oo_{\bbP(1,1,2)}(1)|$ such that $C_0$ is tri-tangent to $\Delta'$, meaning
$$
C_0|_{\Delta'} = 2p_1 + 2p_2 + 2p_3
$$
where $p_1,p_2$ and $p_3$ are general points of $\Delta'$. In other words, $C_0$ is a $6$-gonal curve of genus $16$ such that one member of the $g_6^1$ is a sextuple point, and another member consists of three double points.

Conversely, the normalization of any such $12$-ic curve $C_0 \subset \bbP(1,1,2)$ can be realized as a member of $|-K_\bbP|_{S'}|$ for a K3 surface $S'\in |-K_\bbP|$.
\end{lemma}

\begin{proof}
Using the parametrization of $\mathfrak{f} = \mathrm{Exc}(\varepsilon|_\Sigma)$ as a $\bbP(1,3)$ with coordinates $[u_0:w]$, we know from the equations for $C$ in $\bbP(1,1,2,3)$:
$$
v^3 = u_1h_5(\mathbf u,v,w),\, w^2 = u_0f_5(\mathbf u,v,w),
$$
that the restriction $C|_{\mathfrak f}$ is cut out by an equation of the form $\tau u_0^6 + \lambda u_0^3w + w^2 = 0$, where $\lambda$ and $\tau$ are constants. Hence, the curve $C$ meets the fibre $\mathfrak f$ of $\varepsilon$ at two points, but it does not contain the indeterminacy point $p_w$ of $\varepsilon$. This implies in particular that the restriction $\varepsilon|_C$ is a regular map with image a curve $C_0 \subset \bbP(1,1,2)$. Since $\varepsilon|_\Sigma$ is birational and an isomorphism from $\Sigma - \mathfrak f$ onto its image (\color{purple}Lemma \ref{lem:Sigmabirational}\color{black}), the morphism from $C$ to $C_0$ is birational; besides, it maps the two points $C\cap \mathfrak f$ to a single point $p$. By \color{purple}Lemma \ref{lem:smoothbranches2}\color{black}, the curve $C_0$ has two smooth local branches at $p$, and $C$ is the normalization of $C_0$, since it is smooth.

Furthermore, $C_0$ is a member of $|\mathcal O_{\bbP(1,1,2)}(d)|$ for some $d$ such that
$$
\frac{d}{2} = C_0 \cdot \mathcal O_{\bbP(1,1,2)}(1) = C \cdot \mathcal O_{\bbP(1,1,2,3)}(1) = \mathcal O_{\bbP(1,1,2,3)}(6)^2\cdot \mathcal O_{\bbP(1,1,2,3)}(1) = 6.
$$
The curve $C_0$ is thus given by a degree $d = 12$ equation on $\bbP(1,1,2)$; in other words, it is a sextic section of the cone over a conic.

In particular, let $\Delta = \left\{ u_1 = 0\right\}$ be the line through $p$ in $\bbP(1,1,2)$. By the above, $C_0|_\Delta$ has degree $6$. But the intersection of $C_0$ with $\Delta$ is the image via $\varepsilon$ of $C\cap \left\{ u_1 = 0 \right\}$, and since we have $C\subset \Sigma$ and $\Sigma \cap \left\{ u_1 = 0\right\} = \left\{ u_1 = v = 0\right\} = \mathfrak f$, the curve $C_0$ has only one contact point with $\Delta$:
$$
C_0\cap \Delta = \varepsilon(C|_{u_1 = 0}) \subset \varepsilon(\mathfrak f) = \left\{ p \right\}.
$$
Hence $C_0|_\Delta = 6p$.

Now let $\Delta'$ be the line $\left\{ u_0 = 0\right\}$ in $\bbP(1,1,2)$. The restriction of $C_0$ to $\Delta'$ is the image via $\varepsilon$ of $C|_{u_0=0}$. Consider the surface
$$
\Theta = \left\{ w^2 = u_0f_5(\mathbf u,v,w)\right\}
$$
in $\bbP(1,1,2,3)$; by the fact that $C = \Sigma \cap \Theta$ and 
$$
\Theta|_{u_0 = 0} = (w^2) = 2\ell
$$
where $\ell$ is the curve $u_0=w=0$, which is a $\bbP(1,2)$ with coordinates $[u_1:v]$, and $\Sigma \cdot \ell = 3$ in $\bbP(1,1,2,3)$, we have
$$
C|_{u_0 = 0} = \Theta|_{u_0 = 0} \cap \Sigma|_{u_0 = 0} = 2\ell \cap \Sigma|_{u_0 = 0}
$$
which consists of three double points. Therefore, $C_0|_{\Delta'} = 2p_1 + 2p_2 + 2p_3$ where $p_1,p_1,p_3$ are general points of $\Delta'$. Since $C \to C_0$ is an isomorphism over $C_0 - \left\{ p \right\}$, the three contact points of $C_0$ with $\Delta'$ are tangency points.

Conversely, let $C_0'$ be such a curve in $\bbP(1,1,2)$ and $C'$ its proper transform in $\Sigma$ via the birational map
$$
\varepsilon|_\Sigma : \Sigma \dashrightarrow \bbP(1,1,2).
$$ 
As $C_0'$ is given by an equation of degree $12$, we have
$$
6 = C_0' \cdot \mathcal O_{\bbP(1,1,2)}(1) = C' \cdot \mathcal O_{\bbP(1,1,2,3)}(1) = \Sigma \cdot \mathcal O_{\bbP(1,1,2,3)}(6) \cdot \mathcal O_{\bbP(1,1,2,3)}(1),
$$
i.e., $C' = \mathcal O_{\bbP(1,1,2,3)}(6)|_\Sigma$ in $\mathrm{Pic}(\Sigma)$.

The restriction exact sequence
$$
0 \to \mathcal O_{\bbP(1,1,2,3)} \to \mathcal O_{\bbP(1,1,2,3)}(6) \to \mathcal O_{\bbP(1,1,2,3)}(6)|_\Sigma \to 0
$$
and the vanishing of $h^1(\bbP(1,1,2,3),\mathcal O_{\bbP(1,1,2,3)})$ imply that the map 
$$
H^0(\bbP(1,1,2,3),\mathcal O_{\bbP(1,1,2,3)}(6)) \to H^0(\Sigma,\mathcal O_{\bbP(1,1,2,3)}(6)|_\Sigma)
$$
is surjective, hence there exists a sextic $\Theta'$ of $\bbP(1,1,2,3)$ such that $C' = \Sigma \cap \Theta'$. Let $f_6(\mathbf u,v,w) = 0$ be an equation for $\Theta'$. As  $C_0'$ is tri-tangent to the line $\Delta'$, the set $C_0' \cap \left\{ u_0 = 0\right\}$ has cardinality $3$, whereas
$$
\deg C'|_{u_0 = 0} = \Theta'|_{u_0 = 0} \cdot \Sigma|_{u_0 = 0} = \mathcal O_{\bbP(1,1,2,3)}(6)^2 \cdot \mathcal O_{\bbP(1,1,2,3)}(1) = 6.
$$
Hence $\Theta'|_{u_0 = 0}$ is a nonreduced curve $2\ell'$ of $\bbP(1,2,3)_{[u_1:v:w]}$ such that $\ell' \cdot \Sigma|_{u_0 = 0} = 3$. This yields $f_6(\mathbf u,v,w)|_{u_0 = 0} = h_3(u_1,v,w)^2$ with $h_3$ a homogeneous cubic. Up to scaling, we have $h_3 = w + \alpha(u_1,v)$ with $\deg \alpha = 3$, and the change of variables $w \mapsto w + \alpha(u_1,v)$, which is an automorphism of $\bbP(1,1,2,3)$, yields 
$$
f_6(\mathbf u,v,w)|_{u_0 = 0} = w^2,
$$
thus 
$$
f_6(\mathbf u,v,w) = u_0f'_5(\mathbf u,v,w) + w^2
$$
for some homogeneous quintic $f'_5$. Hence $C'$ lies on the surface $\Theta'$ of equation
$$
u_0f'_5(\mathbf u,v,w) + w^2 = 0.
$$
As a consequence, $C'$ is the complete intersection in $\bbP(1,1,2,3)$ given by the following equations.
$$
u_0f'_5(\mathbf u,v,w) + w^2 = 0,\, v^3 = u_1h_5(\mathbf u,v,w).
$$
As $\bbP$ is cut out in $\bbP(1^2,2,3,5^2)_{[u_0:u_1:v:w:x_0:x_1]}$ by the equations $u_0x_0 = w^2$ and $u_1x_1 = v^3$ (see \color{purple}(\ref{eq:2 3 10 15 model})\color{black}) and the pullbacks of the hyperplanes of $\bbP^{17}$ are the quintics hypersurfaces of $\bbP(1^2,2,3,5^2)$, it makes it visible that $C'$ is a linear section of $\bbP$ in $\bbP^{17}$:
$$
C' = \bbP \cap \left\{ x_0 = -f'_5(\mathbf u,v,w),\, x_1 = h_5(\mathbf u,v,w) \right\}.
$$
In other words, there exists a hyperplane section $S'\in |-K_\bbP|$ of $\bbP$ in $\bbP^{17}$ such that $C'$ is a hyperplane section of $S'$, i.e., $C' \in |-K_\bbP|_{S'}|$.
\end{proof}

\vspace{1cm}

\small I\scriptsize NSTITUT DE \small M\scriptsize ATHÉMATIQUES DE \small T\scriptsize OULOUSE (CNRS UMR 5219), \small U\scriptsize NIVERSITÉ \small P\scriptsize AUL \small S\scriptsize ABATIER, 31062 \small T\scriptsize OULOUSE CEDEX 9, \small F\scriptsize RANCE

\vspace{.2cm}
\normalsize \textit{E-mail address:} \texttt{bruno.dewer@math.univ-toulouse.fr}

\end{document}